\documentclass[12pt, a4paper]{amsart}
\usepackage{graphics,color,hyperref,verbatim}
\usepackage{amsmath}
\usepackage{amsfonts}
\usepackage{amssymb}
\usepackage{amsthm}
\usepackage{graphicx}
\usepackage{psfrag}
\usepackage{marvosym}
\usepackage{mathrsfs}
\usepackage{enumitem}

\usepackage[a4paper, includeheadfoot, margin=2 cm, top=0.8 cm, footskip=0.8 cm]{geometry} 

\usepackage{multicol} 
\usepackage[usenames,dvipsnames]{xcolor} 
\usepackage{tikz, tikz-3dplot, pgfplots}
\usepackage{tikz-cd}
\usetikzlibrary[positioning,patterns] 
\usetikzlibrary{matrix,arrows,decorations.pathmorphing}

\definecolor{light-gray}{gray}{0.7}

\usepackage{ytableau}

\input xy
\xyoption{all}

\usepackage{calligra}
\usepackage{mathrsfs}
\DeclareMathAlphabet{\mathcalligra}{T1}{calligra}{m}{n}
\DeclareFontShape{T1}{calligra}{m}{n}{<->s*[1.5]callig15}{}

\newtheorem{theorem}{Theorem}[section]
\newtheorem*{theoremstar}{Theorem}

\newtheorem{lemma}[theorem]{Lemma}
\newtheorem{proposition}[theorem]{Proposition}

\newtheorem{corollary}[theorem]{Corollary}

\theoremstyle{definition}
\newtheorem{definition}[theorem]{Definition}

\newtheorem{example}[theorem]{Example}

\newtheorem{remark}[theorem]{Remark}
\newtheorem{theorem-definition}[theorem]{Theorem-Definition}
\newtheorem{lemma-definition}[theorem]{Lemma-Definition}

\newtheorem{notation}[theorem]{Notation}

\numberwithin{equation}{section}

\renewcommand{\AA} {\mathbb{A}}

\newcommand{\CC} {\mathbb{C}}

\newcommand{\GG} {\mathbb{G}}

\newcommand{\LL} {\mathbb{L}}

\newcommand{\NN} {\mathbb{N}}

\newcommand{\PP} {\mathbb{P}}
\newcommand{\QQ} {\mathbb{Q}}
\newcommand{\RR} {\mathbb{R}}

\newcommand{\TT} {\mathbb{T}}

\newcommand{\ZZ} {\mathbb{Z}}

\newcommand {\shB} {\mathcal{B}}
\newcommand {\shC} {\mathcal{C}}
\newcommand {\shD} {\mathcal{D}}

\newcommand {\shN} {\mathcal{N}}

\newcommand {\shS} {\mathcal{S}}

\newcommand {\shP} {\mathcal{P}}

\newcommand {\sC} {\mathscr{C}}

\newcommand {\sE} {\mathscr{E}}
\newcommand {\sF} {\mathscr{F}}

\newcommand {\sH} {\mathscr{H}}
\newcommand {\sI} {\mathscr{I}}

\newcommand {\sL} {\mathscr{L}}
\newcommand {\sM} {\mathscr{M}}
\newcommand {\sN} {\mathscr{N}}
\newcommand {\sO} {\mathscr{O}}
\newcommand {\sP} {\mathscr{P}}
\newcommand {\sQ} {\mathscr{Q}}
\newcommand {\sR} {\mathscr{R}}

\newcommand {\sU} {\mathscr{U}}
\newcommand {\sV} {\mathscr{V}}

\newcommand {\foR} {\mathfrak{R}}
\newcommand {\foS} {\mathfrak{S}}

\newcommand {\bG} {\mathbf{G}}

\newcommand {\bdd} {\mathbf{d}}

\usepackage{calligra}
\usepackage{mathrsfs}
\DeclareMathAlphabet{\mathcalligra}{T1}{calligra}{m}{n}
\DeclareFontShape{T1}{calligra}{m}{n}{<->s*[1.5]callig15}{}

\newcommand{\blank}{\underline{\hphantom{A}}}

\newcommand {\D} {\operatorname{D}}

\newcommand{\sExt}{\mathscr{E} \kern -1pt xt}

\renewcommand {\H} {\operatorname{H}}
\newcommand {\Hom} {\operatorname{Hom}}
\newcommand {\sHom}{\mathscr{H}\kern-5pt\mathcalligra{om}}

\newcommand {\id} {\operatorname{id}}
\newcommand {\Id} {\operatorname{Id}}

\renewcommand {\Im} {\operatorname{Im}}

\newcommand {\kk}{\Bbbk}

\newcommand {\Ker} {\operatorname{Ker}}

\newcommand {\Pic} {\operatorname{Pic}}
\newcommand {\Proj} {\operatorname{Proj}}

\newcommand{\pr} {\mathrm{pr}}

\newcommand {\quadand} {\quad\text{and}\quad}

\newcommand {\Spec} {\operatorname{Spec}}

\newcommand {\Sym} {\operatorname{Sym}}

\newcommand{\sTor}{\mathscr{T} \kern -3pt or}

\newcommand{\Hilb} {\operatorname{Hilb}}
\newcommand{\Bl} {\operatorname{Bl}}

\newcommand{\Dqc}{\mathrm{D}_{\mathrm{qc}}}
\newcommand{\Dpc}{\mathrm{D}_{\mathrm{coh}}^-}
\newcommand{\Dltpc}{\mathrm{D}_{\mathrm{coh}}^{\mathrm{b}}}

\newcommand{\Dperf} {\mathrm{D}^{\mathrm{perf}}}

\newcommand{\Ind}{\operatorname{Ind}}
\newcommand{\Fun}{\mathrm{Fun}}

\newcommand{\Map}{\mathrm{Map}}

\newcommand{\CAlgDelta}{\operatorname{CAlg}^\Delta}

\newcommand {\coh} {\mathrm{coh}}

\newcommand{\fib}{\operatorname{fib}}
\newcommand{\cofib}{\operatorname{cofib}}

\newcommand {\dSchur}{\mathbb{S}}
\newcommand {\dWeyl}{\mathbb{W}} 

\makeatletter
\newcommand*\bigcdot{\mathpalette\bigcdot@{.5}}
\newcommand*\bigcdot@[2]{\mathbin{\vcenter{\hbox{\scalebox{#2}{$\m@th#1\bullet$}}}}}
\makeatother

\makeatletter
\newcommand*\bigdot{\mathpalette\bigdot@{.5}}
\newcommand*\bigdot@[2]{\mathbin{\hbox{\scalebox{#2}{$\m@th#1\bullet$}}}}
\makeatother

\newcommand{\Grass}{\operatorname{Grass}}

\newcommand{\Flag}{\operatorname{Flag}}

\newcommand {\Incidence} {\mathfrak{I}{\rm ncid}}

\title[Derived Categories of Derived Grassmannians]{Derived Categories of Derived Grassmannians}
\author[Q.Y.\ JIANG]{Qingyuan Jiang}

\address{School of Mathematics, University of Edinburgh, James Clerk Maxwell Building, Peter Guthrie Tait Road, Edinburgh EH9 3FD, United Kingdom.}
\email{qingyuan.jiang@ed.ac.uk}

\begin{document}

\begin{abstract} 
This paper establishes semiorthogonal decompositions for derived Grassmannians of perfect complexes with Tor-amplitude in $[0,1]$. This result  verifies the author's Quot formula conjecture \cite{J21} and generalizes and strengthens Toda's result in \cite{Tod21b}. 

We give applications of this result to various classical situations such as blowups of determinantal ideals, reducible schemes, and varieties of linear series on curves.

Our approach utilizes the framework of derived algebraic geometry, allowing us to work over arbitrary base spaces over $\QQ$. It also provides concrete descriptions of Fourier-Mukai kernels in terms of derived Schur functors.
\end{abstract}

\maketitle

\section{Introduction}
This paper establishes semiorthogonal decompositions for a broad class of maps $\Grass_X(\sE; d) \to X$, where $\Grass_X(\sE; d)$ is the relative Grassmannian of a complex $\sE$ over $X$ (\cite{J22b}):

\begin{theoremstar}[Theorem \ref{thm:SOD}]
Let $d \in \ZZ_{>0}$. For any scheme (or more generally, prestack) $X$ over $\QQ$, any perfect complex $\sE$ of Tor-amplitude in $[0,1]$ and rank $r \geq 0$ on $X$, and any type of derived category $\D \in \{\Dqc, \Dpc, \Dltpc, \Dperf\}$, there is a semiorthogonal decomposition 
\begin{align}
\label{eqn:intro:SOD}
\D(\Grass_X(\sE;d)) 
&=  \left \langle  \text{$\binom{r}{i}$ copies of } \D(\Grass_X(\sE^\vee[1];d-i))   \right \rangle_{0 \le i \le \min\{r, d\}}.
\end{align}
This semiorthogonal decomposition is induced by faithfully functors $\Phi^{(i,\lambda)}$ (Notation \ref{notation:Phi_lambda}) that are explicitly expressed in terms of derived Schur functors applied to universal perfect complexes on the incidence loci, parametrized by Young diagrams $\lambda$ of height $\le (r-i)$ and width $\le i$. 
\end{theoremstar}

This result verifies and generalizes the author's Quot formula conjecture \cite[Conj. A.5]{J21}. 

Yukinobu Toda  \cite{Tod21b} has established a version of this theorem\footnote{The semiorthogonal decompositions in these two papers have different semiorthogonal orders, but we expect that they differ by a sequence of mutations.} using a different method, the categorified Hall product. His theorem applies to any smooth quasi-projective complex variety $X$. This paper extends and strengthens Toda's result by removing the assumptions  of smoothness and quasi-projectivity on the base $X$, providing explicit descriptions of the Fourier--Mukai kernels, and including the cases for $\D= \Dqc, \Dpc$, and $\Dperf$. 

Our theorem both unifies and generalizes the following important results:

\begin{itemize}
	\item Orlov's projective bundle formula \cite{Orlov92}.
	\item Kapranov's exceptional collections for Grassmannians and the generalization to Grassmannian bundles \cite{K88, BLV, Ef}.
	\item Orlov's blowup formula \cite{Orlov92}. 
	\item The semiorthogonal decompositions for standard flips \cite{BO, Tod21a,BFR,J21}.
	\item Orlov's universal hyperplane section formula \cite{Orlov06}.
	\item The embedding of derived categories for Grassmannian flips (see \cite{BLV2, BLV3, DS, BCF+, J21, Tod21a, Tod21b}). 
	\item Pirozhkov's formula for total spaces of universal bundles on Grassmannians \cite{Pi}.
	\item The author and Leung's projectivization formula (\cite{JL18}; see also \cite{Kuz07, Tod2}).
	\end{itemize}
It not only extends these results to arbitrary base $X$ over $\QQ$ 
but also to the stratified situations.

We first consider the case where $r > d$. Assuming for simplicity that $\sE$ is presented by a generically injective map $\sO_X^m \xrightarrow{\sigma} \sO_X^n$, where $r=n-m \ge 1$, then:
\begin{itemize}
	\item The map $\Grass_X(\sE;d) \to X$ is a {\em stratified} Grassmannian bundle. The general fibers are Grassmannian varieties $\GG_d(r)$, but the fiber dimension jumps over the degeneracy loci $X_j = D_{m-j}(\sigma)$ where the map $\sigma$ has rank $\le m-j$, for $j \ge 1$.
	\item The maps $\Grass_X(\sE^\vee[1]; j) \to X$ on the right-hand side of \eqref{eqn:intro:SOD} are (derived) partial resolutions of the degeneracy loci $X_j=D_{m-j}(\sigma)$, $j \ge 1$. Consequently, their derived categories provide noncommutative partial resolutions of $X_j$'s.	
\end{itemize}

Therefore, in this case, our theorem extends Kapranov's result to stratified Grassmannian bundles. The formula \eqref{eqn:intro:SOD} implies that $\D(\Grass_X(\sE;d))$ contains $\binom{r}{d}$ copies of $\D(X)$, which corresponds to a family version of Kapranov's exceptional collections for a genuine $\GG_d(r)$-bundle. The ``corrections" are given by the noncommutative partial resolution $\D(\Grass_X(\sE^\vee[1]; j)$ of the degeneracy loci $X_j$, $j=1,2,\ldots, d$, capturing the contributions arising from the fiber-dimension-jumping behavior of $\Grass_X(\sE;d)$ over $X_j$.

Similarly, the case $d=r$ of the theorem extends Orlov's blowup formula \cite{Orlov92} for blowups along locally complete intersection (l.c.i.) subschemes to non-l.c.i. cases (see Corollary \ref{cor:SOD:blowup.det}).

In the case where $d>r$, $\D(\Grass_X(\sE;d))$ and $\D(\Grass_X(\sE^\vee[1]; d-r))$ are both noncommutative partial resolution of the degeneracy locus $X_{d-r}$.  We should view $\Grass_X(\sE;d)) \dashleftarrow \Grass_X(\sE^\vee[1]; d-r)$ as a derived generalization of Grassmannian flip. The formula \eqref{eqn:intro:SOD} recovers the embedding of derived categories for this flip, with the orthogonal complement given by noncommutative partial resolutions of higher degeneracy loci $X_{d-r+j}$, $j=1,2,\ldots,r$.

\begin{remark}[Characteristic-zero assumption] 
Notice that in our proof, the characteristic-zero assumption is only required in Lemma \ref{lem:flag:Lascoux}.\eqref{lem:flag:Lascoux-2}. Consequently, if we consider cases where the involved Lascoux-type complexes of Lemma \ref{lem:flag:Lascoux}.\eqref{lem:flag:Lascoux-2} are characteristic-free (e.g., when d = 1), our theorem's result is characteristic-free.
\end{remark}

\subsection{Classical Applications}
Due to the complex behavior of map $\Grass_X(\sE;d) \to X$, our theorem finds applications to various interesting classical situations: 
\begin{itemize}
	\item ({\em Blowup formula for blowups of determinantal ideals \S \ref{sec:blowup.det}}). Let $\Bl_{Z}(X) \to X$ be the blowup of a scheme $X$ along a determinantal subscheme of codimension $(r+1)$ considered in \S \ref{sec:blowup.det}. Then we obtain a semiorthogonal decomposition
\begin{align*}  \D\left(\Bl_{Z}(X)\right) =   \left \langle    \Big \langle \text{$\binom{r}{j}$ copies of~} \D(\widetilde{X_j}) \Big \rangle_{1 \le j \le r} ,  ~ \D(X)  \right \rangle,
\end{align*}
where $\widetilde{X_j}$ are (possibly derived) partial resolutions of the determinantal loci $X_j$, $j=1,\ldots,r$; see Corollary \ref{cor:SOD:blowup.det}.
	\item ({\em Derived categories for reducible schemes \S \ref{sec:reducible}}). In \cite[Examples 7.22]{J22a}, the projectivization formula was used to obtain the following formula for attaching a rational tail $\PP^1$ to a smooth point $p$ of a complex curve $C$:
	\begin{align*}
	\D\big(C \bigsqcup\nolimits_{p} \PP^1\big)   = \big\langle \D(\Spec \CC[\varepsilon_1]), \D(C) \big \rangle,
	\end{align*}
	where $\CC[\varepsilon_1]$ is the derived ring of dual numbers with $\deg (\varepsilon_1) = 1$ and $\varepsilon_1^2=0$; see also \cite[Proposition 6.15]{KS}.  This paper greatly generalizes this result to a large class of reducible schemes (see Corollary \ref{cor:SOD:reducible}), which includes the central fibers of the deformation-to-normal-cone construction as special cases (Remark \ref{rmk:deftonormalcone}).
	\item ({\em Varieties of Linear Series on Curves}). 
	Consider the varieties $G_{d}^{r}(C)$ parametrizing linear series of degree $d$ and dimension $r$ on a smooth complex projective curve of genus $g \ge 1$ (see \cite[Chapters IV, V]{ACGH}). Our theorem implies that $G_{d}^{r}(C)$ have natural derived enhancements $\bG_{d}^{r}(C)$, for which there is a semiorthogonal decomposition
	\begin{align*} 
	\D(\bG_{d}^{r}(C)) = \left\langle
\text{$\binom{1-g+d}{i}$ copies of } \D(\bG_{2g-2-d}^{r-i}(C)) \right\rangle_{0 \leq i \leq \min\{1-g+d, r+1\}}
	\end{align*}
	provided that $d \ge g-1$ and $r \ge -1$; see Corollary \ref{cor:SOD:curves} and \cite[Corollary 1.6]{Tod21b}. This result extends Toda's result for symmetric products of curves in \cite{Tod2} (see also \cite{JL18, BK19}) and the author's result \cite{J21} for the case when $r=1$. 
\end{itemize}

For special curves, the above Corollary \ref{cor:SOD:curves} gives rise to examples of flips of classical threefolds, where the semiorthogonal decomposition contains components given by nonclassical derived schemes (see Example \ref{eg:curves:g=5}). It also produces examples of derived equivalences for threefold flops induced by nonclassical derived incidence schemes (see Example \ref{eg:curves:g=7}). 

The framework presented in this paper allows us to extend the above Corollary \ref{cor:SOD:curves} to families of singular integral curves $\sC/S$, with the role of $\Pic^d(C)$ replaced by the compactified Jacobians $\overline{{\rm Jac}}^d_{\sC/S}$; see Remark \ref{rmk:curves:Hilb}. 

\subsection{A Categorified Decomposition Theorem} 
For a proper map $Y \to X$ between complex algebraic varieties, the Beilinson-- Bernstein--Deligne (BBD) decomposition theorem (see \cite{dCM}) provides a decomposition for intersection cohomologies
	$${\rm IH}^k(Y) \simeq \bigoplus\nolimits_{i} {\rm IH}^{k-d_i}(\overline{X_i}, L_i),$$
where $X_i \subseteq X$ are strata for the map $Y \to X$, $L_i$ are locally systems on $X_i$, and $d_i \in \mathbb{Z}$. 

A fundamental question is in which situation can we ``categorify" this result, in the sense of finding a semiorthogonal decomposition of the
derived category $\D(Y)$ of $Y$, with pieces given by derived categories of spaces supported over the closure $\overline{X_i}$ of strata $X_i$.  For instance, such a categorified decomposition is not possible for $K$-trivial contractions $Y \to X$.

As the spaces  on the right-hand side of the formula \eqref{eqn:intro:SOD} are derived partial resolutions of the closed strata for the map $Y=\Grass_X(\sE;d) \to X$, we could regard our main Theorem \ref{thm:SOD} as such a categorification  for a broad class of maps.

\subsection{Derived Algebraic Geometry}
This paper uses the framework of derived algebraic geometry (DAG), developed by Lurie, To{\"e}n and Vezzosi and many others (\cite{DAG, HTT, SAG, ToenDAG, GRI}). DAG plays a crucial role in this paper in the following aspects:

\begin{enumerate}
	\item ({\em Generality and compatibility with base change}). 
	The theorem applies to any prestack $X$ over $\QQ$, including all (derived) schemes and stacks as special cases. Moreover, the formulation of the formula \eqref{eqn:intro:SOD} commutes with arbitrary base change.
			
	\item ({\em Fourier--Mukai kernels via derived Schur functors}).  	
	The theorem provides explicit descriptions of the Fourier--Mukai kernels involved in terms of derived Schur functors. 
	The derived Schur functors are non-abelian derived functors (in the sense of  Quillen \cite{Qui} and Lurie \cite{HTT}, or equivalently, animations in the sense of \cite[\S 5.1]{CS}) of classical Schur module functors. 
	
	The theory of derived Schur functors has been studied in \cite[\S 3]{J22b}. Importantly, they are highly {\em computable}: using the generalized Illusie's isomorphisms \cite[Proposition 3.34]{J22b}, the derived Schur functors appearing in the theorem can be computed using Akin, Buchsbaum, and Weyman's theory of Schur complexes \cite{ABW, BHL+}.
		
	\item ({\em Derived incidence correspondence schemes}). 
	The Fourier--Mukai kernels of the theorem are supported on certain universal incidence loci (Definition \ref{def:incidence}). These incidence loci generally possess non-trivial derived structures, even in the cases where all involved spaces in the formula  \eqref{eqn:intro:SOD} are classical (see Example \ref{eg:curves:g=7}). They are the derived zero loci of cosections of the form   
		$$\sQ_+^\vee \boxtimes_X \sQ_-^\vee [1] \xrightarrow{\rho_+^\vee \boxtimes \rho_-^\vee[1]} \sE^\vee \boxtimes \sE  \xrightarrow{\rm ev} \sO,$$
	where $\sQ_\pm$ are universal quotient bundles. 
	 This incidence relation can be seen as a higher-rank and shifted version of the universal quadric incidence relation studied in homological projective duality (\cite{Kuz07, Pe19, JLX17}). 
\end{enumerate}

\subsection{Other Related Works}
The Chow-theoretical version of this paper's main theorem has been established by the author in  \cite{J19,J20}. 


In Koseki's paper \cite{Kos21}, Theorem \ref{thm:SOD} (in smooth case, for $\D=\Dltpc$) is used to prove a categorical blow-up formula for Hilbert schemes of points:
	$$\D(\Hilb_n(\widehat{S})) = \Big \langle \text{$p(j)$ copies of $\D(\Hilb_{n-j}(S))$} \Big \rangle_{j=0,1,\ldots,n},$$
where $\widehat{S}$ is the blowup of a smooth complex surface $S$ at a point and $p(j)$ is the number of partitions of $j$.
Koseki \cite{Kos21} also considered the cases of higher rank sheaves on del Pezzo, K3, or abelian surfaces. For general surfaces, the moduli spaces of higher rank sheaves are highly singular. We expect our Theorem \ref{thm:SOD} to be helpful to generalize the above results in these situations and address open questions (1) \&(2)  of \cite[\S 1.4]{Kos21}.

The flag correspondences of relative Grassmannians (see Notation \ref{notation:Psi_k}) have also been studied by Hsu in \cite{H21}. 
We expect the results and methods presented in this paper to be beneficial for investigating the categorical actions explored in {\em loc. cit.}.

In the case where $d=r$ in Theorem \ref{thm:SOD}, the map $\Grass(\sE;r) \to X$ should be regarded as a derived version of blowup, and we expect it to be closely related to the concept of derived blowups studied by Hekking, Khan and Rydh (see \cite{KR18, He21}).


\subsection{Notation and Convention}
We will use the framework of $\infty$-categories developed by Lurie in \cite{HTT}.
Our notations and terminologies will mostly follow those of \cite{J22a, J22b}. Here, we list the notations and conventions that are frequently used in this paper:

\begin{itemize}[leftmargin=*]
	\item ({\em $\infty$-categories of spaces}). We let $\shS$ denote the {\em $\infty$-category of spaces} (or equivalently, the $\infty$-category of $\infty$-groupoids). For a pair of objects $C,D$ in an $\infty$-category $\shC$, we let $\Map_\shC(C,D) \in \shS$ denote their {\em mapping space}. We let $\shC^{\simeq}$ denote {\em core} of $\shC$, that is, the $\infty$-category obtained from $\shC$ by discarding all non-invertible morphisms. For a pair of $\infty$-categories $\shC$ and $\shD$, we let $\Fun(\shC,\shD)$ denote the $\infty$-category of functors from $\shC$ to $\shD$.  
	\item ({\em Simplicial commutative ring}). We let $\CAlgDelta$ denote the $\infty$-category of ``derived rings", that is, {\em simplicial commutative rings} (see \cite[Definition 25.1.1.1]{SAG}; or equivalently, {\em animated commutative rings} in the sense of \cite[\S 5.1]{CS}).
	\item ({\em Prestacks}). A {\em prestack} is a functor $X \colon \CAlgDelta \to \shS$. A map between prestacks $X, Y \colon  \CAlgDelta \to \shS$ is a natural transformation $f \colon X \to Y$ of the functors . The notion of a prestack is probably the most general concept of spaces in algebraic geometry (\cite[Chapter 2 \S 0.1]{GRI}), and includes all derived schemes and derived higher stacks as special cases. 
	\item ({\em Partitions}). We let $B_{\ell,d}$ denote the set of {\em partitions} of height $\le \ell$ and width $\le d$, i.e., partitions $\lambda = (\lambda_1, \lambda_2, \ldots, \lambda_\ell)$ such that $d \ge \lambda_1 \ge \lambda_2 \ge \ldots \ge \lambda_\ell \ge 0$. For a partition $\lambda \in B_{\ell, d}$, we let $|\lambda| = \sum_{i=1}^{\ell} \lambda_i$. We denote its {\em transpose} by $\lambda^{t} = (\lambda_1^t, \lambda_2^t, \ldots, \lambda_s^t)$, i.e., for any $i \in \ZZ_{>0}$, $\lambda^t_i$ is the number of $j$'s such that $\lambda_j \ge i$. By convention, if one of $\ell$ and $d$ is zero, we set $B_{\ell,d}$ to be the singleton of zero partition $(0)$; we let $B_{\ell,d} = \emptyset$ if $\ell<0$ or $d<0$.
 	\item 
	({\em Notations for Derived categories}).
	In this paper, we will use the symbol $\D$ to represent one of the following derived {\em $\infty$-categories}: $\Dqc$, $\Dpc$, $\Dltpc$, or $\Dperf$. Specifically, for a prestack $X$, this paper will consider the following derived $\infty$-categories $\D(X)$: 
	\begin{itemize}
		\item ($\D =\Dqc$). We let $\Dqc(X)$ denote the $\infty$-category of quasi-coherent complexes on $X$ (\cite[\S 3.2]{J22a}) and let $\Dqc(X)^{\le 0}$ denote the full subcategory spanned by complexes $\sE$ such that $\sH^i(\sE):=\pi_{-i}(\sE)=0$ for $i>0$.
		If $X$ is a quasi-compact, quasi-separated scheme, the homotopy category of $\Dqc(X)$ is equivalent to the triangulated derived category of unbounded complexes of $\sO_X$-modules with quasi-coherent cohomologies. 
		\item ($\D=\Dpc, \Dltpc$). We let $\Dpc(X)$ (resp. $\Dltpc(X)$) denote the full subcategory of $\Dqc(X)$ spanned by almost perfect complexes (resp. locally truncated almost perfect complexes); see \cite[Proposition 6.2.5.2]{SAG} (resp., cf. \cite[Notation 6.4.1.1]{SAG}).
		For a Noetherian scheme $X$, the homotopy category of $\Dpc(X)$ (resp. $\Dltpc(X)$) corresponds to the triangulated derived category $\D^-(\coh  (X))$ (resp. $\D^{\rm b}(\coh(X))$) of right-bounded (resp. bounded) complexes of coherent sheaves on $X$, justifying the notations.
		\item ($\D = \Dperf$). We let $\Dperf(X)$ denote the $\infty$-category of perfect complexes on $X$. Then $\Dperf(X)$ is equivalent to subcategory of $\Dqc(X)$ spanned by dualizable objects. 
	\end{itemize}
	\item  
	({\em Tor-amplitude}). 
	A quasi-coherent $R$-complex $M$, where $R \in \CAlgDelta$, is said to have Tor-amplitude in $[0,1]$ if, for any discrete $R$-module $N$, $\pi_i(M \otimes_R N) = 0$ for $i \not\in [0,1]$. 
	A quasi-coherent complex $\sE$ over $X$ is said to have Tor-amplitude in $[0,1]$ if, for any $\eta \colon \Spec R \to X$, where $R \in \CAlgDelta$, $\eta^*(\sE)$ has Tor-amplitude in $[0,1]$ as an $R$-complex. 
	\item  
	({\em Derived convention}). 
	All the functors are assumed to be {\em derived}. For example, if $f \colon X \to Y$ is a map between schemes, $\sE$ is a sheaf on $X$, then $f_*(\sE)$ corresponds to the {\em derived} pushforward $\RR f_*(\sE)$ in the classical convention.
	\item 
	({\em Grothendieck's convention}).
	We will use Grothendieck's convention for projectivizations $\PP(\sE)$, Grassmannians  $\Grass(\sE;d)$ and flags $\Flag(\sE;\bdd)$, so that they parametrize {\em quotients} rather than sub-objects.
	 For example, the projectivization $\PP_X(\sE)$  parametrizes line bundle quotients of $\sE$ over $X$. \end{itemize}

\subsection{Acknowledgment}
The author would like to thank Arend Bayer for numerous helpful discussions and suggestions throughout this project, Richard Thomas for many valuable suggestions on the paper and helpful discussions on relative Grassmannians and degeneracy loci, and Yukinobu Toda for fruitful discussions related to the Quot formula conjecture and helpful comments on an earlier draft of this paper. 
This project originated when the author was a member at IAS, and he would like to thank J{\'a}nos Koll{\'a}r and Mikhail Kapranov for inspiring discussions during that period.
The author is supported by the Engineering and Physical Sciences Research Council [EP/R034826/1] and by the ERC Consolidator grant WallCrossAG, no. 819864.

\section{Derived Grassmannians and Incidence Correspondences}

\subsection{Derived Grassmannians and Derived Schur Functors}
This subsection briefly reviews the theory of derived Grassmannians and of derived Schur functors developed in \cite{J22b}. We let $X$ be a prestack and let $\sE \in \Dqc(X)^{\le 0}$.

\subsubsection{Derived Grassmannians and Derived Flag Schemes}

Let $\bdd= (0 \le d_1< \ldots <d_k)$ be an increasing sequence of integers, where $k \ge 1$. The {\em derived flag scheme of $\sE$ type $\bdd$} (\cite[Definition 4.28]{J22b}) is the prestack over $X$, denoted by
	$$\pr_{\Flag(\sE;\bdd)} \colon \Flag_X(\sE;\bdd) = \Flag(\sE;\bdd) \to X,$$
which carries each $\eta \colon T = \Spec A \to X$, where $A \in \CAlgDelta$, to the full sub-Kan complexes of $\Fun(\Delta^k, \Dqc(T)^{\le 0})^{\simeq}$ spanned by those elements
	$$\zeta_T = (\eta^*\sE \xrightarrow{\varphi_{k,k+1}} \sP_{k} \xrightarrow{\varphi_{k-1,k}} \sP_{k-1} \to \cdots \xrightarrow{\varphi_{1,2}} \sP_1),$$
where each $\phi_{i,i+1}$ is surjective on $\pi_0$, and $\sP_i$ is a vector bundle over $T$ of rank $d_i$, $i \le 1 \le k$. 

Derived flag schemes are derived extensions of Grothendieck's classical flag schemes (\cite[Proposition 4.37]{J22b}). 
The natural projection $\Flag(\sE;\bdd) \to X$ is a relative derived scheme (\cite[Proposition 4.38]{J22b}).
The formation of $\Flag_X(\sE; \bdd)$ commutes with arbitrary derived base change $X' \to X$ (\cite[Proposition 4.32]{J22b}).
If $\sE$ is a perfect complex of Tor-amplitude in $[0,1]$, then the projection $\Flag(\sE;\bdd) \to X$ is a proper, quasi-smooth relative derived scheme, with an invertible relative dualizing complex (see \cite[Corollary 4.47]{J22b}). 
We refer to \cite[\S 4.3]{J22b} for more details of their properties. 

There are two important special cases of derived flag schemes:

\begin{example}[Derived Grassmannians; {\cite[Definition 4.3]{J22b}}]
\label{eg:dGrass}
If $k=1$, $\bdd = (d)$, then we denote the projection $\Flag_X(\sE,\bdd) \to X$ by 
	$$\pr_{\Grass(\sE;d)} \colon \Grass_X(\sE;d) = \Grass(\sE;d) \to X,$$
and refer to it as the {\em rank $d$ derived Grassmannian of $\sE$}. It is by definition an element of $\Fun(\CAlgDelta, \shS)_{/X}$ which carries each $\eta \colon T \to X$, where $T \in \CAlgDelta$, to the space of morphisms 
	$$\left\{ u \colon \eta^* \sE \to \sP \mid \text{$u$ is surjective on $\pi_0$ and $\sP$ is a vector bundle of rank $d$ on $T$} \right\}^{\simeq}.$$ 
We will denote the universal fiber sequence on $\Grass(\sE;d)$ by 
	$$\sR_{\Grass(\sE;d)} \to \pr_{\Grass(\sE;d)}^*(\sE) \xrightarrow{\rho} \sQ_{\Grass(\sE;d)},$$ 
where $\sQ_{\Grass(\sE;d)}$ is the universal quotient bundle of rank $d$.
\end{example}

\begin{example}[Derived Complete Flag Schemes; {\cite[Example 4.31]{J22b}}]
\label{eg:dflag}
Let $n \ge 1$ be a positive integer, let $\bdd = \underline{n} :=(1,2, 3, \cdots, n)$. We will refer to
	$$\pr_{\Flag(\sE;\underline{n})} \colon \Flag_{X}(\sE; \underline{n}) = \Flag(\sE; \underline{n}) \to X.$$
as the {\em derived complete flag scheme of type $n$}. We denote the universal quotient sequence by
	$$\pr_{\Flag(\sE;\underline{n})}^* (\sE) \xrightarrow{\phi_{n,n+1}}  \sQ_{n} \xrightarrow{\phi_{n-1,n}} \sQ_{n-1} \to \cdots \xrightarrow{\phi_{1,2}} \sQ_1,$$
where $\sQ_i$ is the universal quotient bundle of rank $i$, and let 	
	$$\sL_i : =  \fib(\phi_{i-1,i} \colon \sQ_{i} \twoheadrightarrow \sQ_{i-1})$$
denote the universal line bundle on $\Flag(\sE; \underline{n})$, where we set $\sQ_0 = 0$ by convention. Consequently, for each $1 \le i \le n$, we have $\sL_1 \otimes \sL_2 \otimes \cdots \otimes \sL_i \simeq \det \sQ_i$. For any sequence of integers $\lambda = (\lambda_1, \lambda_2, \ldots, \lambda_n) \in \ZZ^n$, we define a line bundle $\sL(\lambda)$ on $\Flag(\sE;\underline{n})$ by the formula:
	\begin{equation*}
	\sL(\lambda) : = \sL_1^{\otimes \lambda_1} \otimes \sL_2^{\otimes \lambda_2} \otimes \cdots \otimes \sL_n^{\otimes \lambda_n}.
	\end{equation*}
If $\sE = \sV$ is a vector bundle of rank $n$, then the morphism $\phi_{n,n+1} \colon \pr_{\Flag(\sE;\underline{n})}^* (\sV) \to \sQ_{n}$ is an equivalence, and the forgetful map induces an equivalence $\Flag(\sV; \underline{n}) \xrightarrow{\simeq} \Flag(\sV; \underline{n-1})$.
\end{example}

\subsubsection{Forgetful Maps Between Derived Flag Schemes}
\label{subsec:forgetful}
If $\bdd' = (d_{i_j})_{1 \le j \le \ell}$ be a subsequence of $\bdd = (d_i)_{1 \le i \le k}$, where $(i_1 < i_2 < \ldots < i_\ell)$ is a subsequence of $\underline{n}=(1,2, 3, \cdots, n)$, then there is a natural forgetful map (\cite[\S 4.3.2]{J22b})
	$$\pi_{\bdd',\bdd} \colon \Flag(\sE; \bdd) \to \Flag(\sE; \bdd').$$
We will need the following proposition, which is a special case of \cite[Corollary 4.34]{J22b}:

\begin{proposition}[{\cite[Corollary 4.34]{J22b}}]
\label{prop:dflag:forget}
Let $k \ge 2$, let $i$ be an integer such that $1 \le i \le k-1$, and let $\bdd' = (d_1, \cdots, d_{i})$ and $\bdd''=(d_{i+1}, \cdots, d_{k})$ so that $\bdd = (\bdd', \bdd'')$.
\begin{enumerate}[leftmargin=*]
	\item \label{prop:dflag:forget-1}
	The forgetful map $\pi_{\bdd'',\bdd} \colon \Flag_X(\sE; \bdd) \to \Flag_X(\sE;  \bdd'')$ identifies $\Flag_X(\sE; \bdd)$ as the derived flag scheme $\Flag\big(\sQ_{d_{i+1}}^{(\bdd'')}; \bdd'\big)$ over $\Flag_X(\sE;  \bdd'')$, where $\sQ_{d_{i+1}}^{(\bdd'')}$ is the universal quotient bundle on $\Flag_X(\sE;  \bdd'')$ of rank $d_{i+1}$.
	\item \label{prop:dflag:forget-2}
	The forgetful map $\pi_{\bdd',\bdd}  \colon \Flag_X(\sE; \bdd) \to \Flag_X(\sE; \bdd')$ identifies $\Flag_X(\sE; \bdd)$ as the derived flag scheme $\Flag\big(\fib(\varphi_i^{(\bdd')}); d_{i+1}-d_{i}, \ldots, d_{k} - d_{i}\big)$ over $\Flag_X(\sE; \bdd')$, where $\sQ_{d_{i}}^{(\bdd')}$ is the universal rank $d_i$ quotient bundle on $\Flag_X(\sE; \bdd')$ and  $\varphi_i^{(\bdd')} \colon \pr_{\Flag(\sE; \bdd')}^*(\sE) \to \sQ_{d_{i}}^{(\bdd')}$ is the universal quotient map.
	\end{enumerate}
\end{proposition}

\begin{remark} 
\label{rmk:dflag.fibersqure}
As a direct consequence of Proposition \ref{prop:dflag:forget}, if we assume $k \ge 3$, let $i,j$ be integers such that $1 \le i < j < k$ and let $\bdd$ be written as:
		$$\bdd = (\underbrace{d_1, \cdots, d_{i}}_{\bdd^{(1)}}; \underbrace{d_{i+1}, \cdots, d_{j}}_{\bdd^{(2)}}; \underbrace{d_{j+1}, \cdots, d_{k}}_{\bdd^{(3)}}),$$
then the natural commutative square of forgetful maps
		$$
	\begin{tikzcd}[column sep = 5 em]
		\Flag_X(\sE; \bdd)  \ar{d}{\pi_{(\bdd^{(1)}, \bdd^{(2)}), \bdd}}  \ar{r}{\pi_{(\bdd^{(2)}, \bdd^{(3)}), \bdd}}&  \Flag_X(\sE; \bdd^{(2)}, \bdd^{(3)}) \ar{d}{\pi_{\bdd^{(2)}, (\bdd^{(2)}, \bdd^{(3)})}} \\
		 \Flag_X(\sE; \bdd^{(1)}, \bdd^{(2)})  \ar{r}{\pi_{\bdd^{(2)}, (\bdd^{(1)}, \bdd^{(2)})}}  &  \Flag_X(\sE; \bdd^{(2)})
	\end{tikzcd}
	$$
is a pullback square (i.e., it is a derived fiber product square).
\end{remark}

\subsubsection{Derived Schur Functors}
Derived Schur functors, studied in \cite{J22b}, are non-abelian derived functors (in the sense of Dold--Puppe \cite{DP}, Quillen \cite{Qui} and Lurie \cite{HTT}, or equivalently, animations in the sense of Cesnavicius and Scholze \cite[\S 5.1]{CS}) of the classical Schur module functors. Specifically, for a prestack $Y$ and a partition $\lambda$, the {\em derived Schur functor associated with $\lambda$}, as defined in \cite[Definition 3.5, \S 3.5.2]{J22b}, is an endorfunctor of the derived $\infty$-category of quasi-coherent complexes denoted by
	$$\dSchur^{\lambda} \colon \Dqc(Y)^{\le 0} \to \Dqc(Y)^{\le 0}.$$
This functor extends the classical Schur module functors of vector bundles and preserves sifted colimits. 
In the particular case where $\lambda = (n)$ (resp. $\lambda =\underbrace{(1, \ldots, 1)}_{n \,\text{terms}}$) for some integer $n \ge 0$,  the derived Schur functor $\dSchur^{(n)} = \Sym^{n}$ (resp. $\dSchur^{(1, \ldots, 1)} = \bigwedge^{n}$) corresponds to the {\em $n$th derived symmetric power} (resp. {\em $n$th derived exterior power}) functor studied by Dold--Puppe \cite{DP}, Illusie \cite{Ill} and Lurie \cite[\S 3.1]{DAG} \cite[\S 25.2]{SAG}.  

The derived Schur functors possess many desirable functorial properties, such as their compatibility with arbitrary base change, and they satisfy derived generalizations of classical formulae like Cauchy's decomposition formula and Littlewood--Richardson rule. We refer the readers to \cite[\S 3]{J22b} for a more comprehensive discussion and detailed explanations. 

\subsubsection{Derived Borel--Weil--Bott Theorem}
The theories of derived flag schemes and derived Schur functors are connected via a derived generalization of the Borel--Weil--Bott theorem.

\begin{theorem}[Borel--Weil--Bott Theorem for Derived Complete Flag Schemes; {\cite[Theorems 5.26, 5.35]{J22b}}]
\label{thm:BWB}
Consider the situation described in Example \ref{eg:dflag}, assume that $\sE$ is a perfect complex rank $n$ and Tor-amplitude in $[0,1]$ over $X$, and let $\lambda = (\lambda_1, \ldots, \lambda_n) \in \NN^n$. Then:
\begin{enumerate}
	\item
	\label{thm:BWB-1}
	 If $\lambda$ is a partition, we have a canonical equivalence $(\pr_{\Flag(\sE;\underline{n})})_*(\sL(\lambda)) \simeq \dSchur^{\lambda}(\sE)$, where $\dSchur^{\lambda}(\sE)$ is the derived Schur functor applied to the perfect complex $\sE$.
	\item 
	\label{thm:BWB-2}
	If $X$ is defined over $\QQ$, one of the following two mutually exclusive cases occurs:
	\begin{enumerate}
	\item There exists a pair of integers $1 \le i < j \le n-1$ such that $\lambda_i  - \lambda_j = i - j$. In this case, $(\pr_{\Flag(\sE;\underline{n})})_* (\sL(\lambda)) \simeq 0.$
	\item There exists a unique permutation $w \in \foS_n$ such that $w \bigdot \lambda$ is non-increasing. In this case,  there is a canonical equivalence
		$(\pr_{\Flag(\sE;\underline{n})})_* (\sL(\lambda)) \simeq \dSchur^{w \bigdot \lambda}(\sE) [- \ell(w)].$
		Here, $w \bigdot \lambda = w(\lambda+\rho) -\rho$ denotes the dot action, and $\ell(w)$ is the length of $w$.
	\end{enumerate}
\end{enumerate}
\end{theorem}

In the special case where  $\sE$ is a vector bundle, the above results reduce to the familiar Borel--Weil--Bott theorem for vector bundles; see \cite{Dem} and \cite[Theorems 4.1.4 \&  4.1.10]{Wey}.

\begin{remark}[{\cite[Corollaries 5.30 \& 5.39]{J22b}}]
The above theorem implies that corresponding Borel--Weil--Bott theorem for derived Grassmannians $\Grass(\sE;d)$ studied in Example \ref{eg:dGrass}. Specifically, let $d$ be an integer such that $1 \le d \le n$. Let $\alpha = (\alpha_1, \ldots , \alpha_{d})$ and $\beta = (\beta_1, \ldots, \beta_{n-d})$ be two partitions and let $\lambda = (\alpha,\beta)$ be their concatenation. Then $\sV(\alpha, \beta) := (\pi_{(d), \underline{n}})_*(\sL(\lambda))\simeq \dSchur^{\alpha}(\sQ_{\Grass(\sE;d)}) \otimes \dSchur^{\beta}(\sR_{\Grass(\sE;d)})$.
Consequently, the theorem implies the following:
\begin{enumerate}
	\item If $\lambda$ is a partition, $(\pr_{\Grass(\sE;d)})_* (\sV(\alpha,\beta)) \simeq \dSchur^{\lambda}(\sE)$.
	\item If $X$ is defined over $\QQ$, then one of the following two mutually exclusive cases occurs:
	\begin{enumerate}
	\item There exists a pair of integers $1 \le i < j \le n-1$ such that $\lambda_i  - \lambda_j = i - j$. In this case, $(\pr_{\Grass(\sE;d)})_* (\sV(\alpha,\beta))\simeq 0.$
	\item There exists a unique permutation $w \in \foS_n$ such that $w \bigdot \lambda$ is non-increasing. In this case,  there is a canonical equivalence
		$(\pr_{\Grass(\sE;d)})_* (\sV(\alpha,\beta)) \simeq \dSchur^{w \bigdot \lambda}(\sE) [- \ell(w)].$
	\end{enumerate}
\end{enumerate}
\end{remark}

\subsection{Incidence Correspondences}
This subsection studies the incidence correspondences between derived Grassmannians, generalizing the incidence correspondences of the projectivization case \cite[\S 7.1]{J22a}.
Throughout this subsection, we let $X$ be a prestack and  assume that $\sE$ is a perfect complex over $X$ of Tor-amplitude in $[0,1]$ of rank $r \ge 0$. Notice that the shifted dual $\sE^\vee[1]$ is also a perfect complex of Tor-amplitude in $[0,1]$, but has rank $(-r)$.

\begin{definition}[Incidence Correspondences]
\label{def:incidence}
Let $(d_+, d_-) \in \NN^2$ be a pair of integers and consider the derived Grassmannians (Example \ref{eg:dGrass}):
	$$\pr_+ \colon \Grass(\sE;d_+) \to X \quad \text{and} \quad \pr_- \colon \Grass(\sE^\vee[1]; d_-) \to X,$$
with tautological fiber sequences $\sR_\pm \to \pr_\pm^*(\sE) \xrightarrow{\rho_\pm} \sQ_\pm$, where $\sQ_{\pm}$ are universal quotient bundles of rank $d_\pm$, respectively. 
We define the {\em universal incidence locus $\Incidence_{(d_+,d_-)}(\sE)$} to be the derived zero locus of the cosection of the perfect complex $\sQ_+^\vee \boxtimes_X \sQ_-^\vee [1]$ over $\Grass_{d_+}(\sE) \times_X \Grass_{d_-}(\sE^\vee[1])$ defined as the composition
	$$\sQ_+^\vee \boxtimes_X \sQ_-^\vee [1] \xrightarrow{\rho_+^\vee \boxtimes \rho_-^\vee[1]} \pr_+^*(\sE^\vee) \boxtimes_X \pr_-^*(\sE) \xrightarrow{\rm ev} \sO_{ \Grass(\sE;d_+) \times_X \Grass(\sE^\vee[1]; d_-)}.$$
(Here, for complexes $\sM_+$ on $\Grass(\sE;d_+)$ and $\sM_-$ on $\Grass(\sE^\vee[1];d_-)$, we use $\sM_+ \boxtimes_X \sM_-$ to denote the external tensor product $\pr_{1}^* \sM_+ \otimes_{\sO_Z} \pr_{2}^* \sM_-$, where $\pr_i$ are the projections from $\Grass(\sE;d_+) \times_X \Grass(\sE^\vee[1]; d_-)$ to its $i$th factors.)
We will refer to the commutative diagram
	\begin{equation} \label{diagram:Inc:Corr}
	\begin{tikzcd}[row sep= 2.5 em, column sep = 4 em]
		\Incidence_{(d_+,d_-)}(\sE)  \ar{rd}{\widehat{\pr}}
		\ar{r}{r_+}  \ar{d}[swap]{r_-} &  \Grass(\sE;d_+) \ar{d}{\pr_+}\\
		\Grass(\sE^\vee[1];d_-) \ar{r}{\pr_-} & X
	\end{tikzcd}
	\end{equation}
as the {\em incidence diagram}. 
By construction, there is a canonical commutative diagram
	$$
	\begin{tikzcd}
		r_-^*(\sQ_{-}^\vee) [1] \ar{d} \ar{r}{r_-^*(\rho_-^\vee [1])} & \widehat{\pr}^*(\sE) \ar{d}{r_+^*(\rho_+)} \\
		0 \ar{r} & r_+^*(\sQ_+)
	\end{tikzcd}
	$$
in $\Dperf(\Incidence_{(d_+,d_-)}(\sE))$. We consider the following perfect complex 
	$$\sE^{\rm univ}_{(d_+,d_-)} :=\cofib \left(r_-^*(\sQ_{-}^\vee) [1] \to \fib\Big(\widehat{\pr}^*(\sE) \xrightarrow{r_+^*(\rho_+)} r_+^*(\sQ_+) \Big) \right),$$ 
and refer to it as the {\em universal perfect complex} on $\Incidence_{(d_+,d_-)}(\sE)$.
\end{definition}

\begin{example} 
\begin{enumerate}
	\item If $d_-=0$, $\Incidence_{(d_+,0)}(\sE) = \Grass(\sE;d_+)$ and $\sE^{\rm univ}_{(d_+,0)} = \sR_{\Grass(\sE;d_+)}$.
	\item In the universal local situation of Notation \ref{notation:setup:local} (where $X = \underline{\Hom}_{\kk}(\kk^m,\kk^n)$ and $\sE = [\sO_X^m \xrightarrow{\tau} \sO_X^n]$ is the tautological map), the perfect complex $\sE^{\rm univ}_{(d_+,d_-)}$ is canonically represented by a universal two-term complex of vector bundles  
	$\Big[r_-^*\big(R_{\GG_{d_-}^{-}}^\vee\big) \to r_+^*\big(R_{\GG_{d_+}^{+}}\big)\Big]$.
\end{enumerate}
\end{example}

\begin{lemma}
\label{lem:E^univ:incidence}
In the situation of Definition \ref{def:incidence}, we have:
\begin{enumerate}
	\item 
	\label{lem:E^univ:incidence-1}
	There is a canonical equivalence
		$$\sE^{\rm univ}_{(d_+,d_-)} 
		\simeq \fib\left( 
		\cofib \Big(r_-^*(\sQ_{-}^\vee) [1] \xrightarrow{r_-^*(\rho_-^\vee [1])} \widehat{\pr}^*(\sE) \Big) \to  r_+^*(\sQ_+) 
		\right).		
		$$ 
	\item 
	\label{lem:E^univ:incidence-2}
	If $\sE$ has rank $r$ (and Tor-amplitude in $[0,1]$), then $\sE^{\rm univ}_{(d_+,d_-)}$ is a perfect complex over $\Incidence_{(d_+,d_-)}(\sE)$ of Tor-amplitude in $[0,1]$ and rank $(r-d_+ + d_-)$.
	
\end{enumerate}
\end{lemma}

\begin{proof}
To prove assertion \eqref{lem:E^univ:incidence-1}, consider the following induced commutative diagram
	$$
	\begin{tikzcd}
		r_-^*(\sQ_{-}^\vee) [1]  \ar{d} \ar{r} & \fib(r_+^*(\rho_+)) \ar{d} \ar{r}  &\widehat{\pr}^*(\sE) \ar{d} \\
		0 \ar{r} & \sE^{\rm univ}_{(d_+,d_-)}\ar{r}  \ar{d}& \cofib(r_-^*(\rho_-^\vee [1])) \ar{d} \\
		& 0 \ar{r} & r_+^*(\sQ_+) 
	\end{tikzcd}
	$$
where all three squares are pushouts, hence bi-Cartesian as $\Dqc(\Incidence_{(d_+,d_-)}(\sE))$ is a stable $\infty$-category. This proves \eqref{lem:E^univ:incidence-1}. Since $\cofib(r_-^*(\rho_-^\vee [1]))$ has Tor-amplitude in $[0,1]$ and rank $(r+d_-)$, $r_+^*(\sQ_+)$ is a vector bundle of rank $d_+$, and the natural map $\cofib(r_-^*(\rho_-^\vee [1])) \to r_+^*(\sQ_+)$ is surjective on $\pi_0$, assertion \eqref{lem:E^univ:incidence-2} follows from \eqref{lem:E^univ:incidence-1} (see \cite[Proposition 7.2.4.23.(2)]{HA}). 
\end{proof}

\begin{lemma}
\label{lem:dGrass:incidence}
In the situation of Definition \ref{def:incidence}, we have:
\begin{enumerate}
	\item 
	\label{lem:dGrass:incidence-1}
	The projection $r_+$ of \eqref{diagram:Inc:Corr} identifies $\Incidence_{(d_+,d_-)}(\sE))$ as the rank $d_-$ derived Grassmannian of the perfect complex $\cofib \Big(r_+^*(\rho_+^\vee[1])  \colon \sQ_+^{\vee}[1] \to \sE^\vee[1]\Big)$ over $\Grass(\sE;d_+)$. 
	\item 
	\label{lem:dGrass:incidence-2}
	The projection $r_-$ of \eqref{diagram:Inc:Corr} identifies $\Incidence_{(d_+,d_-)}(\sE))$ as the rank $d_+$ derived Grassmannian of the perfect complex $\cofib\Big(r_-^*(\rho_-^\vee[1]) \colon \sQ_-^{\vee}[1] \to \sE\Big)$ over $\Grass(\sE^\vee[1]; d_-)$. 
\end{enumerate}
Consequently, the projections $r_\pm$ are both proper, quasi-smooth relative derived schemes. 
\end{lemma}

\begin{proof}
Similar to the projectivization case \cite[Lemma 7.3]{J22a}, assertions \eqref{lem:dGrass:incidence-1} and \eqref{lem:dGrass:incidence-2} follow from the characterizations of closed immersions of the form $\Grass(\sF''; d) \to \Grass(\sF;d)$ between derived Grassmannians induced by cofiber sequences $\sF' \to \sF \to \sF''$ of connective complexes (see \cite[Proposition 4.19]{J22b}).
As a result, the assertion about the properness and quasi-smoothness of the maps $r_\pm$ follow from \cite[Corollary 4.28]{J22b}.
\end{proof}

\begin{remark}[Expected Dimensions and Classical Criteria]
\label{rmk:classical.criteria}
In the situation of Lemma \ref{lem:dGrass:incidence}, assume $X$ has constant dimension $\dim X$, then the quasi-smooth relative derived schemes $\Grass(\sE;d_+)$, $\Grass(\sE^\vee[1];d_-)$ and $\Incidence_{(d_+,d_-)}(\sE)$ over $X$ have virtual dimensions 
	$$\dim X + d_+ (r - d_+), \quad \dim X  -d_- (r +d_-) \quadand \dim X + (d_+ - d_-) + d_+ d_- - d_+^2 - d_-^2,$$
respectively.
If $X$ is a Cohen--Macaulay scheme, then $\Grass(\sE;d_+)$, (resp. $\Grass(\sE^\vee[1];d_-)$, $\Incidence_{(d_+,d_-)}(\sE)$)) is classical if and only if its underlying classical scheme has dimension equal to its virtual dimension (see the proof of \cite[Lemma 6.7]{J21}).
Moreover, if all these schemes are classical, then $\Incidence_{(d_+,d_-)}(\sE)$ is canonically isomorphic to the classical fiber product of $\Grass(\sE;d_+)$ and $\Grass(\sE^\vee[1],d_-)$ over $X$.
\end{remark}

\subsection{Compatibility of Incidence and Flag Correspondences}

\begin{proposition}
\label{prop:forget:incidence.flag}
In the situation of Definition \ref{def:incidence}, we let $d_+' > d_+$ be another integer, then there is a canonical forgetful map
	$$
	{\rm forg} \colon \Flag(\sE; d_+, d_++1, \ldots, d_+') \times_{\Grass(\sE;d_+')} \Incidence_{(d_+', d_-)}(\sE) \to \Incidence_{(d_+,d_-)}(\sE)
	$$
which identifies the domain of ${\rm forg}$ with the derived flag scheme
 	$$\Flag_{\Incidence_{(d_+,d_-)}(\sE)}\left(\sE^{\rm univ}_{(d_+,d_-)}; 1, 2, \ldots, d_+'-d_+\right)$$ 
of the universal perfect complex $\sE^{\rm univ}_{(d_+,d_-)}$ over $\Incidence_{(d_+,d_-)}(\sE)$.
\end{proposition}

\begin{proof}
Let $\sF : = \cofib\left(\sQ_{-}^\vee[1] \xrightarrow{r_-^*(\rho_-^\vee[1])} \pr_-^*(\sE)\right)$ over $Y := \Grass(\sE^\vee[1]; d_-)$, and consider the following commutative diagram:
	$$
	\begin{tikzcd}
			Z:=\Flag_Y(\sF; d_+, d_++1, \ldots, d_+') \ar{d}{\pi} \ar{r}{\pi'}& \Grass_Y(\sF;d_+') 
			 \ar{d}{\pr'} \\
		 \Grass_{Y}(\sF; d_+) 
		  \ar{r}{\pr} & Y=\Grass(\sE^\vee[1]; d_-),
	\end{tikzcd}
	$$
where the maps $\pi, \pi'$ are the natural forgetful maps between derived flag schemes (\S \ref{subsec:forgetful}), and $\pr, \pr'$ are the natural projections. By virtue of Lemma \ref{lem:dGrass:incidence}, there are canonical equivalences 
	$$\Grass_Y(\sF;d_+') \simeq \Incidence_{(d_+', \,d_-)}(\sE) \quad \text{and} \quad \Grass_{Y}(\sF; d_+) \simeq \Incidence_{(d_+,d_-)}(\sE).$$
Let $r_+ \colon \Incidence_{(d_+, \,d_-)}(\sE)  \to \Grass(\sE; d_+)$ (resp. $r_+' \colon \Incidence_{(d_+', \,d_-)}(\sE)  \to \Grass(\sE; d_+')$) denote the natural projection, and $\sQ_+$ (resp. $\sQ_+'$) the tautological quotient bundle of rank $d_+$ (resp. $d_+'$) over $\Grass(\sE; d_+)$ (resp. $\Grass(\sE; d_+')$). 
By virtue of Proposition \ref{prop:dflag:forget}.\eqref{prop:dflag:forget-1}, the forgetful map $\pi'$ identifies $Z$ as the derived flag scheme 
	$$Z \simeq \Flag_{\Incidence_{(d_+',d_-)}(\sE)}\left(r_+'^*(\sQ_+'); d_+, d_++1, \ldots, d_+' \right).$$ 
Since Proposition \ref{prop:dflag:forget}.\eqref{prop:dflag:forget-1}) also implies that the forgetful map $\Flag(\sE; d_+, d_++1, \ldots, d_+') \to \Grass(\sE; d_+')$ is equivalent to the derived flag bundle of $\sQ_+'$ of type $(d_+, d_++1, \ldots, d_+')$ over $\Grass(\sE; d_+)$, we obtain that $\pi'$ identifies $Z$ with the domain of the map ${\rm forg}$.
On the other hand, by virtue of Proposition \ref{prop:dflag:forget}.\eqref{prop:dflag:forget-2} and the equivalence $\sE^{\rm univ}_{(d_+,d_-)} \simeq \fib\left(\pr^* (\sF) \to r_+^*(\sQ_+)\right)$ of Lemma \ref{lem:E^univ:incidence}.\eqref{lem:E^univ:incidence-1}, the forgetful map $\pi$ identifies $Z$ as the derived flag scheme of $\sE^{\rm univ}_{(d_+,d_-)}$ of type $(1, 2, \ldots, d_+'-d_+)$ over the incidence space $\Incidence_{(d_+,d_-)}(\sE)$. Hence the proposition is proved.
\end{proof}

\begin{corollary}
\label{cor:forg:FlagtoIncidence}
Let $\sE$ be a perfect complex of rank $r \ge 0$ and Tor-amplitude in $[0,1]$, let $d \ge 0$ and $0 \le i \le \min\{d, r\}$ be integers, and let $\lambda = (\lambda_1 \ge  \ldots \ge  \lambda_{r-i})$ be a partition. Then there is a natural forgetful map
	\begin{equation*}
	{\rm forg} \colon \Flag(\sE; d, d+1, \ldots, d+r-i) \times_{\Grass(\sE;d+r-i)} \Incidence_{(d+r-i,d-i)}(\sE) \to \Incidence_{(d,d-i)}(\sE).
	\end{equation*}
which is a proper, quasi-smooth relative derived scheme and induces a canonical equivalence
	$${\rm forg}_{*} \big(\sL_{d+1}^{\lambda_1} \otimes \sL_{d+2}^{\lambda_2} \otimes \ldots \otimes \sL_{d+r-i}^{\lambda_{r-i}}\big) \simeq \dSchur^{\lambda}(\sE^{\rm univ}_{(d,d-i)}),$$
where $\sQ_i$'s are universal quotient bundles of rank $i$ (where $d \le i \le d+r-i$) and $\sL_i = \fib(\sQ_{i} \to \sQ_{i-1})$ are the associated universal line bundles (where $d+1 \le i \le d+r-i$).
\end{corollary}

\begin{proof}
Apply Proposition \ref{prop:forget:incidence.flag} and Theorem \ref{thm:BWB}.\eqref{thm:BWB-1} to the case where $(d_+
,d_-) = (d,d-i)$ and $d_+'=d+r-i$. In this case, $\sE_{(d,d-i)}^{\rm univ}$ is a perfect complex of Tor-amplitude in $[0,1]$ and rank $(r-i)$ (Lemma \ref{lem:E^univ:incidence}.\eqref{lem:E^univ:incidence-2}), and the properness and quasi-smoothness of the map ${\rm forg}$ follow from \cite[Corollary 4.28]{J22b}.
\end{proof}

The above relationship between moduli prestacks yields compatibility result for the induced Fourier--Mukai functors, which we will now investigate. 

\begin{notation}
\label{notation:Phi_lambda}
Assume that we are in the situation of Definition \ref{def:incidence} and let maps $r_\pm$ be defined as in diagram \eqref{diagram:Inc:Corr}. Assume that $r- d_+ + d_- \ge 0$, and let $\lambda = (\lambda_1 \ge \cdots \ge \lambda_{r-d_+ + d_-})$ be a partition and $i \in \ZZ$. We consider Fourier--Mukai functors: 
	\begin{align*}
	\Phi_{(d_+,d_-)}^{\lambda} &= r_{+\, * }\big( r_-^{*}(\blank) \otimes \dSchur^{\lambda}(\sE^{\rm univ}_{(d_+,d_-)}) \big) \colon &\D(\Grass(\sE^\vee[1]; d_-)) \to \D(\Grass(\sE; d_+)). \\
	\Phi_{(d_+,d_-)}^{(i, \lambda)} &= \Phi_{(d_+,d_-)}^{\lambda}(\blank) \otimes \det(\sQ_+)^i \colon & \D(\Grass(\sE^\vee[1]; d_-)) \to \D(\Grass(\sE; d_+)).
	\end{align*}
We will omit the subindex $(d_+,d_-)$ and write $\Phi^{\lambda}$ and $\Phi^{(i,\lambda)}$ instead when there is no confusion. Here, we use the symbol $\D$ to denote any of the following derived $\infty$-categories: $\Dqc$, $\Dpc$, $\Dltpc$ or $\Dperf$. This definition will be justified by the following lemma.
\end{notation}

\begin{lemma} 
\label{lem:Phi_lambda:adjoints}
In the situation of Notation \ref{notation:Phi_lambda}, let $\D = \Dqc$, then the functor $\Phi^{\lambda}$ (resp. $\Phi^{(i,\lambda)}$) admits both a left adjoint $(\Phi^{\lambda})^L$ (resp. $(\Phi^{(i,\lambda)})^L$) and a right adjoint $(\Phi^{\lambda})^R$ (resp. $(\Phi^{(i,\lambda)})^R$). 
Furthermore, all these functors preserve (almost) perfect complexes and locally truncated almost perfect complexes, and commute with arbitrary base change $X' \to X$.
\end{lemma}

\begin{proof}
The left and right adjoints of $\Phi^\lambda$ can be given explicitly by the formula
	\begin{align*}
	(\Phi^{\lambda})^L &= r_{-\,!} \left(r_+^*(\blank) \otimes  \dSchur^{\lambda}(\sE^{\rm univ}_{(d_+,d_-)})^\vee\right) \colon & \D(\Grass(\sE; d_+)) \to \D(\Grass(\sE^\vee[1]; d_-)). \\
	(\Phi^{\lambda})^R & = r_{-\,*} \left(r_+^!(\blank) \otimes  \dSchur^{\lambda}(\sE^{\rm univ}_{(d_+,d_-)})^\vee\right) \colon & \D(\Grass(\sE; d_+)) \to \D(\Grass(\sE^\vee[1]; d_-)). 
\end{align*}
Here, $r_{-\,!}$ denotes the left adjoint of $r_{-}^*$, and $r_{+}^!$ denotes the right adjoint of $r_{+\,*}$. Since $r_\pm$ are proper and quasi-smooth (Lemma \ref{lem:dGrass:incidence}), the desired assertions follow from Lipman--Neeman--Lurie's version of Grothendieck duality (see \cite[Theorem 3.7.(3)]{J22a}).
\end{proof}

\begin{notation}
\label{notation:Psi_k}
Let $\sE$ be a perfect complex of Tor-amplitude in $[0,1]$ and rank $r \ge 1$, and let $d$ be an integer $0 \le d \le r-1$. We consider the following commutative diagram
	\begin{equation}\label{diag:corr:flagd,d+1}
	\begin{tikzcd}[column sep = 4 em, row sep = 2.5 em]
		\Flag(\sE;d,d+1) \ar{d}{p_{+}} \ar{r}{p_{-}} 
		 & \Grass(\sE;d+1) \ar{d}{\pr_{\Grass(\sE;d+1) }} \\
		\Grass(\sE;d) \ar{r}{\pr_{\Grass(\sE;d) }} & X,
	\end{tikzcd}
	\end{equation}
where $p_\pm$ are the natural forgetful maps (\S \ref{subsec:forgetful}). Let $\D$ denote $\Dqc$, $\Dpc$, $\Dltpc$ or $\Dperf$. We consider Fourier--Mukai functors: 
\begin{align*}
	&\Psi_{\sL_{d+1}^k}  = p_{+\, *} (p_-^*(\blank) \otimes \sL_{d+1}^k) \colon &\D(\Grass(\sE;d+1)) \to \D(\Grass(\sE;d)). \\
	&\Psi_{k}  = p_{+\, *} \, p_-^*(\blank) \otimes \det(\sQ_{d})^{\otimes k} \colon 
	& \D(\Grass(\sE;d+1)) \to \D(\Grass(\sE;d)). 
\end{align*}
Here $\sQ_i$'s are universal quotient bundles of rank $i$ for $i=d,d+1$, and $\sL_{d+1} = \fib(\sQ_{d+1} \to \sQ_{d})$. 
\end{notation}

\begin{remark}
\label{rmk:flag:p_pm}
Proposition \ref{prop:dflag:forget} implies that the projection $p_+$ of diagram \eqref{diag:corr:flagd,d+1} identifies $\Flag(\sE;d,d+1)$ as the derived projectivization of the perfect complex $\sR_{\Grass(\sE;d)}$ over $\Grass(\sE;d)$ with $\sO_{p_+}(1) \simeq \sL_{d+1}$, where $\sR_{\Grass(\sE;d)} = \fib\left(\pr_{\Grass(\sE;d)}^*(\sE) \to \sQ_{\Grass(\sE;d)}\right)$ has Tor-amplitude in $[0,1]$ and rank $(r-d)$. The projection $p_{-}$ identifies $\Flag(\sE; d,d+1)$ as the derived projectivization $\PP(\sQ_{\Grass(\sE; d+1)}^\vee)$ of the rank $(d+1)$ vector bundle $\sQ_{\Grass(\sE;d+1)}^\vee$ over $\Grass(\sE; d+1)$, with $\sO_{p_{-}}(1) \simeq \sL_{d+1}^\vee$; or equivalently, as the derived Grassmannian $\Grass(\sQ_{\Grass(\sE;d+1)}; d)$ of $\sQ_{\Grass(\sE;d+1)}$ over $\Grass(\sE; d+1)$, with universal quotient bundle $\sQ_d$.
\end{remark}


As a consequence of Corollary \ref{cor:forg:FlagtoIncidence}, we have the following compatibility result for the Fourier--Mukai functors considered in Notations \ref{notation:Phi_lambda} and \ref{notation:Psi_k}:

\begin{corollary}
\label{cor:forg:FM:Incidence}
In the situation of Corollary \ref{cor:forg:FlagtoIncidence}, let $\D$ denote $\Dqc$, $\Dpc$, $\Dltpc$ or $\Dperf$, assume that $\lambda = (i \ge \lambda_1 \ge  \ldots \ge  \lambda_{r-i} \ge 0) \in B_{r-i,i}$, 
and let 
	$$\Phi^{(i, \lambda)}_{(d,d-i)} \colon \D(\Grass(\sE^\vee[1]; d-i)) \to \D(\Grass(\sE; d)) $$
denote the functor defined in Notation \ref{notation:Phi_lambda} in the case $(d_+,d_-) = (d,d-i)$. Let $\Psi_k$'s be defined as in Notation \ref{notation:Psi_k}. Then there is a canonical equivalence of functors:
	$$\Phi^{(i, \lambda)}_{(d,d-i)}  \simeq  \Psi_{i-\lambda_1} \circ \cdots \circ \Psi_{\lambda_{r-1} - \lambda_{r-i-1}} \circ \det(\sQ_{d+r-i})^{\lambda_{r-i}} \circ  \Phi_{(d+r-i,d-i)}^{(0)}.$$
\end{corollary}

\begin{proof}
For each $d + 1 \le k \le d+r-i$, we have $\det(\sQ_k) \simeq \sL_k \otimes \det(\sQ_{k-1})$. By induction, we obtain a canonical equivalence of functors from $\D(\Grass(\sE; d+r-i))$ to $\D(\Grass(\sE;d))$:
	$$(\otimes \det(\sQ_d)^i )\circ \Psi_{\sL_{d+1}^{\lambda_1}} \circ \cdots \circ \Psi_{\sL_{d+r-i}^{\lambda_{r-i}}} \simeq \Psi_{i-\lambda_1} \circ \cdots \circ \Psi_{\lambda_{r-1} - \lambda_{r-i-1}} \circ (\otimes \det(\sQ_{d+r-i})^{\lambda_{r-i}}),$$
where $\Psi_{\sL_{d+1}^k}$'s are defined in Notation \ref{notation:Psi_k}. Consequently, it suffices to prove that there is a canonical equivalence of functors
	$$\Phi^{\lambda}_{(d,d-i)} \simeq \Psi_{\sL_{d+1}^{\lambda_1}} \circ \cdots \circ \Psi_{\sL_{d+r-i}^{\lambda_{r-i}}} \circ \Phi_{(d+r-i,d-i)}^{(0)} \colon \D(\Grass(\sE^\vee[1]; d-i)) \to \D(\Grass(\sE; d)).$$
	Consider the following commutative diagram:
	$$
	\begin{tikzcd}[column sep = {15em,between origins}]
		& \Flag(\sE; d, d+1, \ldots, d+r-i) \times_{\Grass(\sE;d+r-i)} \Incidence_{(d+r-i,d-i)}(\sE) \ar{d}{\rm forg} \ar{ldd}[swap]{\pi_-} \ar{rdd}{\pi_+} &  \\
		 & \Incidence_{(d,d-i)}(\sE) \ar{ld}{r_-} \ar{rd}[swap]{r_+} & \\
		\Grass(\sE^\vee[1]; d-i) & & \Grass(\sE; d),
	\end{tikzcd}
	$$
where the vertical map ${\rm forg}$ is the forgetful map in Corollary \ref{cor:forg:FlagtoIncidence} and $r_\pm$ are the projection maps of the incidence diagram \eqref{diagram:Inc:Corr}. 
By repeated use of Remark \ref{rmk:dflag.fibersqure}, we see that the composite functor $\Psi_{\sL_{d+1}^{\lambda_1}} \circ \cdots \circ \Psi_{\sL_{d+r-i}^{\lambda_{r-i}}} \circ \Phi_{(d+r-i,d-i)}^{(0)}$ is equivalent to the functor
	\begin{align*}
	&\pi_{+\,*} \left(\pi_-^*(\blank) \otimes (\sL_{d+1}^{\lambda_1} \otimes \sL_{d+2}^{\lambda_2} \otimes \ldots \otimes \sL_{d+r-i}^{\lambda_{r-i}})\right) \\
	&\simeq r_{+ \,*} \circ {\rm forg}_{*} \left( {\rm forg}^* \circ r_-^*(\blank) \otimes  (\sL_{d+1}^{\lambda_1} \otimes \sL_{d+2}^{\lambda_2} \otimes \ldots \otimes \sL_{d+r-i}^{\lambda_{r-i}})\right) \\
	&\simeq r_{+ \,*} \left(r_-^*(\blank) \otimes  {\rm forg}_{*}  (\sL_{d+1}^{\lambda_1} \otimes \sL_{d+2}^{\lambda_2} \otimes \ldots \otimes \sL_{d+r-i}^{\lambda_{r-i}})\right) \\
	& \simeq r_{+\,*} \left(r_-^*(\blank) \otimes \dSchur^{\lambda}(\sE^{\rm univ}_{(d,d-i)}) \right) = \Phi_{(d,d-i)}^{\lambda} (\blank),
	\end{align*}
where the second equivalence follows from projection formula, and the third equivalence follows from Corollary \ref{cor:forg:FlagtoIncidence}. Hence the corollary is proved.
\end{proof}

\section{Semiorthogonal Decompositions of Derived Grassmannians}

This section establishes the main result of this paper, namely Theorem \ref{thm:SOD}. 
In \S \ref{sec:SOD}, we present the result and reduce it to the universal local situation, which will be addressed in \S \ref{sec:local}.

\subsection{Semiorthogonal Decompositions}
\label{sec:SOD}
Let $r \ge 0$ be an integer and let $0 \le i \le r$, so that $B_{r-i, i}$ denotes the set of partitions $\lambda = (i \ge \lambda_1 \ge \lambda_2 \ge \cdots \lambda_{r-i} \ge 0)$. 

\begin{notation}
\label{notation:differene.order}
We define a total order $<$ on the set $\shP_r = \{ (i, \lambda) \mid 0 \le i \le r, \lambda \in B_{r-i,i}\}$ as follows: for any $(i,\lambda), (j,\mu) \in \shP_r$, we write $(i, \lambda) < (j, \mu)$, if $(i-\lambda_1, \lambda_1 - \lambda_2, \ldots, \lambda_{i-1} - \lambda_{i}, \lambda_{i}) <_{\rm lex} (j-\mu_1, \mu_1 - \mu_2, \ldots, \mu_{j-1} - \mu_j, \mu_j)$ in the lexicographical order $<_{\rm lex}$. 
\end{notation}

Notice that if $i=j$, then for $\lambda,\mu \in B_{r-i,i}$, we have $(i,\lambda) < (i, \mu)$ if and only if $\lambda >_{\rm lex} \mu$ (that is, $\lambda$ is smaller than $\mu$ in the {\em opposite} lexicographical order of partitions in $B_{r-i,i}$).

The goal of this section is to establish the following theorem:

\begin{theorem}[Semiorthogonal Decompositions of Derived Grassmannians]
\label{thm:SOD}
Let $X$ be a prestack defined over $\QQ$, let $\sE$ be a perfect complex of Tor-amplitude $[0,1]$ and rank $r \ge 0$ over $X$, let $d \ge 1$ be an integer, and let $\D$ be either $\Dqc$, $\Dpc$, $\Dltpc$ or $\Dperf$.
For any integer $0 \le i \le \min\{r, d\}$ and any partition $\lambda \in B_{r-i,i}$, we let $\Phi^{(i, \lambda)} = \Phi^{(i,\lambda)}_{(d,d-i)}$ denote the functor defined in Notation \ref{notation:Phi_lambda} in the case $(d_+,d_-) = (d,d-i)$, that is:
	\begin{align*}
	&\Phi^{(i, \lambda)} \colon \D(\Grass(\sE^\vee[1];d-i)) \to \D(\Grass(\sE;d)) \\
	&\sF \mapsto r_{+\,*} \big(r_-^*(\sF) \otimes \dSchur^{\lambda}(\sE^{\rm univ}_{(d,d-i)}) \big)\otimes \det (\sQ_{\Grass(\sE;d)})^{i}.
\end{align*}
Then $\Phi^{(i, \lambda)}$ is fully faithful.  Moreover, these functors $\Phi^{(i, \lambda)}$, where $0 \le i \le \min\{r, d\}$ and $\lambda \in B_{r-i,i}$, induce a semiorthogonal decomposition
	\begin{align*}
	\D(\Grass(\sE;d)) 
	=
 	\Big \langle \Im \big(\Phi^{(i,\lambda)} \big)
	\mid 0 \le i \le \min\{r, d\}, \, \lambda \in B_{r-i,i}
	\Big \rangle,
\end{align*}
with semiorthogonal order given by the total order $<$ defined in Notation \ref{notation:differene.order}. Specifically, $\Map\left(\Im (\Phi^{(i,\lambda)}), \Im (\Phi^{(j,\mu)})\right) \simeq 0$ if $(j,\mu) <(i,\lambda)$.
\end{theorem}

Notice that when fixing $d$, the subindices ``$(d,d-i)$" of $\Phi^{(i,\lambda)}_{(d,d-i)}$ are uniquely determined by the superscripts ``$(i,\lambda)$". Therefore, there is no ambiguity in writing $\Phi^{(i, \lambda)}$ for $\Phi^{(i,\lambda)}_{(d,d-i)}$.

\begin{example}
	If $r=4$ and $d \ge 4$, we have a semiorthogonal decomposition 
	\begin{multline*}
		\D(\Grass(\sE;d)) = 
		\Big \langle
	\Im \big(\Phi^{(0,(0))}\big), \Im \big(\Phi^{(1,(1,1,1))}\big), \Im \big(\Phi^{(1,(1,1))}\big),  \Im \big(\Phi^{(2,(2,2))}\big), \\
	\Im \big(\Phi^{(1,(1))}\big), \Im \big(\Phi^{(2,(2,1))}\big), \Im \big(\Phi^{(2,(2))}\big), \Im \big(\Phi^{(3,(3))}\big), \Im \big(\Phi^{(1,(0))}\big), 
	 \Im \big(\Phi^{(2,(1,1))}\big), 
	 \\
	 \Im \big(\Phi^{(2,(1))}\big), \Im \big(\Phi^{(3,(2))}\big), \Im \big(\Phi^{(2,(0))}\big), \Im \big(\Phi^{(3,(1))}\big), \Im \big(\Phi^{(3,(0))}\big), \Im \big(\Phi^{(4,(0))}\big)
	 \Big \rangle.
		\end{multline*}
\end{example}

\begin{proof}[Proof of Theorem \ref{thm:SOD}, Part 1]
Let $\Phi_1, \Phi_2, \ldots, \Phi_N$ be all the functors in 
	$$\{\Phi^{(i,\lambda)} \mid 0 \le i \le \min\{r, d\}, \lambda \in B_{r-i,i}\}$$
listed in ascending order with respect to the total order $<$ on the superscrips $(i,\lambda)$, where $N = \sum_{i=0}^{\min\{r,d\}} \binom{r}{i}$. For each $1 \le j \le N$, we let $\foR_j$ denote the endofunctor $\fib(\id \to \Phi_j \circ \Phi_j^L)$ of $\D(\Grass(\sE;d))$, where $\Phi_j^L$ denotes the left adjoint of $\Phi_j$. Consequently, there is a canonical filtered sequence in $\Fun\left(\D(\Grass(\sE;d)), \D(\Grass(\sE;d))\right)$
	$$\foR_N \circ \foR_{N-1} \circ \cdots \circ \foR_1 \to \foR_{N-1} \circ \cdots \circ \foR_1  \to \cdots \to \foR_2 \circ \foR_1 \to \foR_1 \to \id,$$
where $\cofib(\foR_1 \to \Id) \simeq \Phi_1 \circ \Phi_1^L \colon \D(\Grass(\sE;d))) \to \Im(\Phi_1)$, and for each $2 \le j \le N$, 
	$$\pr_j : = \cofib\left(\foR_j \circ \foR_{j-1} \circ \cdots \circ \foR_1 \to \foR_{j-1} \circ \cdots \circ \foR_1\right) \simeq \Phi_j \circ \Phi_j^L \circ (\foR_{j-1} \circ \cdots \circ \foR_1)$$ 
defines a functor from $\D(\Grass(\sE;d))$ to $\Im(\Phi_j)$. 

Therefore, to establish the desired semiorthogonal decomposition, it is equivalent to prove the following assertions about the functors $\Phi_j$, $\Phi_j^L$ and $\foR_j$ (for the corresponding category $\D$):

\begin{enumerate}[ref=\alph*,label=(\alph*)]
	\item 
	\label{cond:thm:SOD:ff}
	Fully-faithfulness: For each $1 \le j \le N$, the counit map $\Phi_j^L \circ \Phi_j \to \id$ is an equivalence.
	\item 
	\label{cond:thm:SOD:so}
	Semiorthogonality: For all $1 \le j < k \le N$, $\Phi_j^L \circ \Phi_k \simeq 0$.
	\item 
	\label{cond:thm:SOD:gen}
	Generation: $\foR_N \circ \foR_{N-1} \circ \cdots  \circ \foR_1 \simeq 0$. 
\end{enumerate}
We have the following observations:
\begin{itemize}
	\item Since the assertions \eqref{cond:thm:SOD:ff}, \eqref{cond:thm:SOD:so} and \eqref{cond:thm:SOD:gen}, regarded as properties for the pair $(X, \sE)$, are local with respect to Zariski topology, we may assume that $X$ is a derived affine scheme $\Spec A$, where $A \in \CAlgDelta$, and $\sE$ is the cofiber of a map $\sigma \colon A^m \to A^n$ between finite local free sheaves, where $m,n \ge 0$ are integers such that $n-m=r$.
	\item Given that the functors $\Phi_j, \Phi_j^L$, and hence $\foR_j$, preserve all small colimits and perfect objects (Lemma \ref{lem:Phi_lambda:adjoints}), and that for any quasi-compact, quasi-separated derived scheme $Y$, we have $\Dqc(Y) \simeq \Ind(\Dperf(Y))$, to prove the assertions \eqref{cond:thm:SOD:ff}, \eqref{cond:thm:SOD:so} and \eqref{cond:thm:SOD:gen} in the case where $\D = \Dqc$ and $X = \Spec A$, it suffices to verify them in the case where $\D = \Dperf$ and $X = \Spec A$, and vice versa. 
	\item Since the functors $\Phi_j, \Phi_j^L$, and hence $\foR_j$, preserve almost perfect objects and locally truncated almost perfect objects (Lemma \ref{lem:Phi_lambda:adjoints}), the assertions \eqref{cond:thm:SOD:ff}, \eqref{cond:thm:SOD:so} and \eqref{cond:thm:SOD:gen} in the cases $\D = \Dpc$ and $\D=\Dltpc$ can be deduced from the case $\D = \Dqc$.
\end{itemize}
Consequently, it suffices to consider the case where $\D = \Dperf$, $X = \Spec A$, where $A \in \CAlgDelta$, and $\sE = \cofib(\sigma \colon A^m \to A^n)$, where $n-m = r$. Moreover, as the formation of the assertions \eqref{cond:thm:SOD:ff}, \eqref{cond:thm:SOD:so} and \eqref{cond:thm:SOD:gen} commutes with arbitrary base change (as discussed in Lemma \ref{lem:Phi_lambda:adjoints}), we may assume that $X = \Hom_\kk(\kk^m,\kk^n)$ is the total space of $\kk$-homomorphisms from $\kk^m$ to $\kk^n$ and $\sE \simeq [\sO_X^m \xrightarrow{\tau} \sO_X^n]$, where $\kk$ is an ordinary commutative ring and $\tau$ is the tautological map; in this case, the desired assertions are established in the next subsection (see \S \ref{subsec:proof:SOD}).
\end{proof}

\begin{remark}[Dual Exceptional Sequences]
\label{rmk:dual.exc.seq}
Let $d<r$ and consider the functors $\Phi^{(i,\lambda)}$ in the case $i=d$. Then Theorem \ref{thm:SOD} implies that the collection of objects
	$$\{\dSchur^{\lambda}\big(\sR_{\Grass(\sE;d)}\big) \mid \lambda \in B_{r-d,d}\}$$
forms a relative exceptional sequence (\cite[Definition 6.8]{J21}) over $X$ with respect to the opposite lexicographic order. On the other hand, in \cite{J21} we show that 
	$$\{\dSchur^{\alpha}\big(\sQ_{\Grass(\sE;d)} \big)\mid \alpha \in B_{d,r-d}\}$$
forms a relative exceptional sequence over $X$ with respect to the colexicographic order. Using derived Borel--Weil--Bott Theorem \ref{thm:BWB} and the filtered sequences associated with Schur complexes, we can show that these two relative exceptional sequences are dual (and mutation-equivalent) to each other. Details will appear in a separate note.

\end{remark}

\begin{example} 
Let $X \to S$ be any quasi-smooth map between prestacks that is a relative derived algebraic space of constant dimension $r \ge 0$. Then the relative cotangent complex $\LL_{X/S}$ has perfect-amplitude in $[0,1]$ and rank $r$. We let $\TT_{X/S}[1] = \LL_{X/S}^\vee[1]$ denote the shifted tangent complex. Theorem \ref{thm:SOD} implies semiorthogonal decompositions for all $d \ge 0$:
\begin{align*}
\D(\Grass_X(\LL_{X/S};d)) 
=  \left \langle  \text{$\binom{r}{i}$ copies of } \D(\Grass_X(\TT_{X/S}[1];d-i))   \right \rangle_{0 \le i \le \min\{r, d\}}.
\end{align*}
In the special case where $d=r$, the derived relative scheme $\Grass_X(\LL_{X/S};r) \to X$ is closely related to the construction of Nash blowups. 
\end{example}
 
\subsection{The Universal Local Situation}
\label{sec:local}
In this subsection, we let $\D = \Dperf$. 

\subsubsection{The Setup for the Universal Local Situation}
\label{subsec:setup:local}
Now we introduce the basic setup for the universal local situation:

\begin{notation}
\label{notation:setup:local}
\begin{enumerate}
	\item 
	Let $\kk$ be a commutative ring, let $n, m, d \ge 0$ be integers such that $n-m=: r \ge 0$, and $W = \kk^m$ and $V = \kk^n$. 
	\item 
	For a pair of non-negative integers $(d_+, d_-)$, we let 
	$$\GG_{d_+}^+ := \Grass_{\Spec \kk}(V; d_+) \quad \text{and} \quad \GG_{d_-}^- := \Grass_{\Spec \kk}(W^\vee ; d_-)$$
denote the rank $d_\pm$ Grassmannian $\kk$-schemes of $V$ and $W^\vee$, respectively, and let 
	$$R_{\GG_{d_+}^{+}} \hookrightarrow V  \otimes \sO_{\GG_{d_+}^{+}} \twoheadrightarrow Q_{\GG_{d_+}^{+}} \quad \text{and} \quad R_{\GG_{d_-}^{-}} \hookrightarrow W^\vee  \otimes \sO_{\GG_{d_-}^{-}} \twoheadrightarrow Q_{\GG_{d_-}^{-}},$$
denote the tautological short exact sequences, where $Q_{\GG_{d_\pm}^{\pm}}$ are tautological quotient bundles of ranks $d_\pm$, respectively. 
	\item Let $X = \underline{\Hom}_{\kk}(W,V) = \Spec (\Sym_{\kk}^*(W \otimes_{\kk} V^\vee) )$ denote affine $\kk$-space parametrizing $\kk$-homomorphisms from $W$ to $V$, and let $\tau \colon W \otimes \sO_X \to V \otimes \sO_X$ denote the tautological morphism. Let $\sE = [W \otimes \sO_X \xrightarrow{\tau} V \otimes \sO_X]$ (with $V \otimes \sO_X$ placed in degree $0$); then $\sE^\vee[1] \simeq  [V^\vee \otimes \sO_X \xrightarrow{\tau^\vee} W^\vee \otimes \sO_X]$ (with $W^\vee \otimes \sO_X$ placed in degree $0$).
\end{enumerate}
\end{notation}
 
\begin{lemma}
\label{lem:setup:local}
In the situation of Notation \ref{notation:setup:local}, we have canonical identifications:
\begin{align*}
& q_{\GG_{d_+}^+} \colon \Grass(\sE;d_+) \simeq \Spec \Big(\Sym^*_{ \sO_{\GG_{d_+}^{+}}} \big(W \otimes R_{\GG_{d_+}^{+}}^\vee \big) \Big) \to \GG_{d_+}^+.\\
& q_{\GG_{d_-}^-} \colon
\Grass(\sE^\vee[1];d_-)  \simeq \Spec \Big(\Sym^*_{ \sO_{\GG_{d_-}^{-}}} \big(R_{\GG_{d_-}^{-}}^\vee \otimes V^\vee \big) \Big)\to \GG_{d_-}^-.\\
& q_{\Incidence_{(d_+,d_-)}} \colon \Incidence_{(d_+,d_-)}(\sE)  \simeq \Spec \Big(\Sym_{\sO_{\GG_{d_+}^{+} \times \GG_{d_-}^-}}^*\big(R_{\GG_{d_-}^{-}}^\vee \boxtimes R_{\GG_{d_+}^{+}}^\vee \big) \Big) \to \GG_{d_+}^{+} \times \GG_{d_-}^-.
\end{align*}
\end{lemma}
\begin{proof}
In this case, these derived schemes are classical (Remark \ref{rmk:classical.criteria}). Therefore, the desired result follows easily from definitions (see \cite[Lemma 4.1]{BLV3} or \cite[Lemma 5.1]{J21}).
\end{proof} 

\begin{notation}
\label{notation:abuse:local}
In the situation of Lemma \ref{lem:setup:local}, 
to reduce the burden of notations, for objects $\sF_\pm \in \D(\GG_{d_\pm}^\pm)$, we will simply use the {\em same} notations $\sF_+ = q_{\GG_{d_+}^+}^*(\sF_+) \in \D( \Grass(\sE;d_+))$ and $\sF_- =q_{\GG_{d_-}^-}^*(\sF_-) \in \D( \Grass(\sE^\vee[1];d_-))$ to denote their respective pullbacks. 
\end{notation}

\begin{notation}
\label{notation:gen:local}
If $\shC$ is an idempotent-complete stable $\infty$-category (such as $\D(\Grass(\sE;d_+))$ and $\Grass(\sE^\vee[1];d_-)$), and $\{C_i\}_{i \in I}$ is a collection of objects of $\shC$, we let $\langle \{C_i\}_{i \in I}  \rangle\subseteq \shC$ denote the stable subcategory thickly generated by $\{C_i\}_{i \in I}$ (i.e., $\langle \{C_i\}_{i \in I} \rangle$ is the smallest idempotent-complete stable $\infty$-subcategory of $\shC$ which contains all $C_i$). 
\end{notation}

\begin{lemma} 
\label{lem:generators}
In the situation of Lemma \ref{lem:setup:local} and using Notations \ref{notation:abuse:local}, \ref{notation:gen:local}, we have 
	\begin{align*}
\D(\Grass(\sE;d_+)) &= \Big \langle \big \{ \dSchur^{\lambda} \big(R_{\GG_{d_+}^{+}}^\vee\big) \big \}_{\lambda \in B_{n-d_+,d_+}} \Big \rangle\\
	\D(\Grass(\sE^\vee[1];d_-) )&= \Big \langle  \big\{ \dSchur^{\mu} \big(R_{\GG_{d_-}^{-}}\big) \big\}_{\mu \in B_{m-d_-,d_-}} \Big \rangle.
\end{align*}
\end{lemma}
\begin{proof}
This follows from Kapranov's exceptional collections for $\D(\GG_{d_\pm}^{\pm})$ (\cite{K88}; see also \cite{BLV, Ef} for the characteristic-free version) and the fact that the natural projections $q_{\GG_{d_\pm}^\pm}$ are relative affine spaces.
\end{proof}

\subsubsection{Incidence Correspondences in the Universal Local Situation}
This subsection considers the incidence diagram \eqref{diagram:Inc:Corr} of Definition \ref{def:incidence} in the universal local situation \S \ref{subsec:setup:local}. We assume that $d \ge r=m-n$ and consider the incidence diagram \eqref{diagram:Inc:Corr} in the case where $(d_+,d_-) = (d, d-r)$. Let $\Phi = \Phi^{(0)}_{(d,d-r)}$ be the functor of Notation \ref{notation:Phi_lambda} in the case where $\lambda = (0)$, and let $\Phi^L$ be its left adjoint functor, that is:\begin{align*}
	\Phi & = r_{+*} r_{-}^* \colon \D(\Grass(\sE^\vee[1]); d-r) \to \D(\Grass(\sE; d)) \\
	\Phi^L &  = r_{-!} r_{+}^* \colon \D(\Grass(\sE;d)) \to  \D(\Grass(\sE^\vee[1];d-r)).
\end{align*}

\begin{lemma}
\label{lem:flip}
 In the above situation, we have canonical equivalences
	\begin{align*}
	&\Phi\Big(\dSchur^{\lambda} \big(R_{\GG_{d-r}^-}\big)\Big) \simeq \dSchur^{\lambda} \big(R_{\GG_{d}^+}^\vee\big) \quad \text{for all} \quad \lambda \in B_{n-d,d-r}. \\
	&\Phi^{L}\Big(\dSchur^{\lambda} \big(R_{\GG_{d}^+}^\vee \big)\Big) \simeq 
	\dSchur^{\lambda} \big(R_{\GG_{d-r}^-}\big)
	 \quad \text{for all} \quad \lambda \in B_{n-d,d}.
	\end{align*}
\end{lemma}
\begin{proof}
This a special case of the key lemma \cite[Lemma 5.6]{J21}; we present here a characteristic-free proof for readers' convenience. We only prove the first equivalence; the other case is similar. The projection $r_+$ factorizes through a composite map (\cite[Proposition 4.19]{J22b})
	$$\Incidence_{(d,d-r)}(\sE) \xrightarrow{\iota} \GG_{d-r}^- \times_{\kk} \Grass(\sE;d) \xrightarrow{\pr} \Grass(\sE;d),$$
where $\iota$ is a closed immersion induced by a regular section of the vector bundle $Q_{\GG_{d-r}^{-}} \boxtimes R_{\GG_{d}^+}$, and $\pr$ is the canonical projection. Therefore, we have a canonical equivalence
	$$\Phi\Big(\dSchur^{\lambda} \big(R_{\GG_{d-r}^-}\big)\Big) \simeq \pr_*\Big(\dSchur^{\lambda} \big(R_{\GG_{d-r}^-}\big) \otimes \iota_{*} (\sO_{\Incidence_{(d,d-r)}(\sE)}) \Big),$$
where $\iota_{*} (\sO_{\Incidence_{(d,d-r)}(\sE)})$ is resolved by a Koszul complex whose $\ell$th terms are given by
	$$\bigwedge\nolimits^{\ell} \Big(Q_{\GG_{d-r}^{-}}^\vee \boxtimes R_{\GG_{d}^+}^\vee\Big) \quad \text{where} \quad 0 \le \ell \le (n-d)(d-r).$$
By Cauchy's decomposition formula  (\cite[Theorems III.1.4]{ABW}, \cite[Propositions 2.3]{Kou}), there is a canonical filtration of $\bigwedge\nolimits^{\ell} \big(Q_{\GG_{d-r}^{-}}^\vee \boxtimes R_{\GG_{d}^+}^\vee\big)$ whose associated graded is given by 
	$\dSchur^{\mu^t}(Q_{\GG_{d-r}^{-}})^\vee \otimes \dSchur^{\mu}(R_{\GG_{d}^+}^\vee),$
where $\mu$ run through all elements of $B_{n-d,d-r}$ such that $|\mu| = \ell$. Consequently, it suffices to prove the following
	$$\Hom_{\D(\GG_{d-r}^{-})}\Big(\dSchur^{\mu^t}\big(Q_{\GG_{d-r}^{-}}\big), \dSchur^{\lambda} \big(R_{\GG_{d-r}^-}\big) [\lambda]\Big) \simeq \delta_{\mu, \lambda} \cdot \id,$$
where $\delta_{\mu,\lambda} = 1$ if $\mu=\lambda$ and $\delta_{\mu,\lambda} = 0$ if $\mu \neq \lambda$. This follows from that $\big\{(\dSchur^{\mu^t}\big(Q_{\GG_{d-r}^{-}}\big)\big\}_{\mu \in B_{n-d,d-r}}$ and $\big\{\dSchur^{\lambda} \big(R_{\GG_{d-r}^-}\big) [\lambda]\big\}_{\lambda \in B_{n-d,d-r}}$ are dual full exceptional collections of $\D(\GG_{d-r}^{-})$; see \cite[Theorem 7.5]{BLV} and \cite[Theorem 1.6]{Ef}.
\end{proof}

\begin{corollary}
\label{cor:ff:incidence}
In the situation of Lemma \ref{lem:flip}, the functor $\Phi$ is fully faithful, with essentially image $\Im \Phi = \Big\langle \big\{ \dSchur^{\mu} \big(R_{\GG_{d}^+}^\vee\big) \big\}_{\mu \in B_{n-d,d-r}} \Big \rangle \subseteq \D(\Grass(\sE;d))$. 
\end{corollary}

\begin{proof}
Lemma \ref{lem:flip} implies that the counit map $\Phi^L \circ \Phi \to \id$ is an equivalence when evaluated at the generators $\dSchur^{\lambda} \big(R_{\GG_{d-r}^-}\big)$ of $\D(\Grass(\sE^\vee[1]; d-r))$ described in Lemma \ref{lem:generators}, where $\lambda \in B_{n-r,r-d}$. Since the collection of objects $\sF$, for which the counit map $\Phi^L \circ \Phi(\sF) \to \sF$ is an equivalence, forms an idempotent-complete stable $\infty$-subcategory of $\D(\Grass(\sE^\vee[1]; d-r))$, it follows that the counit map $\Phi^L \circ \Phi \to \id$ is an equivalence. Hence the corollary follows.
\end{proof}

\subsubsection{Flag Correspondences in the Universal Local Situation}
Now we consider flag correspondences \eqref{diag:corr:flagd,d+1} in the universal local situation of \S \ref{subsec:setup:local}.
Let $d$ be an integer such that $0 \le d \le n-1$, and let $\Psi=\Psi_0$ be the functor defined in Notation \ref{notation:Psi_k} and let $\Psi^L$ be its left adjoint; that is:
	\begin{align*}
	\Psi &= p_{+\,*} \, p_{-}^* \colon \D(\Grass(\sE; d+1)) \to \D(\Grass(\sE;d)), \\
	\Psi^L & = p_{-\,!}  \,p_{+}^* \colon \D(\Grass(\sE;d)) \to  \D(\Grass(\sE;d+1)),
\end{align*}
where $p_\pm$ are defined as in \eqref{diag:corr:flagd,d+1}, and $p_{-!}$ denotes the left adjoint of $p_{-}^*$. 

The following is analogous to \cite[Lemma 5.6]{J21} in the case where $\ell_+ - \ell_- = 1$; the combinatorics of the Lascoux-type complexes $F_*$ in this case are also similar to that of the staircase complexes studied in \cite{Fon, DS}.

\begin{lemma}
\label{lem:flag:Lascoux}
 In the above situation, we have:
\begin{enumerate}
	\item 
	\label{lem:flag:Lascoux-1}
	For any $\lambda \in B_{n-d,d}$, there is a canonical equivalence 
		$$\Psi^L\Big(\dSchur^{\lambda} \big(R_{\GG_{d}^+}^\vee\big)\Big)
		\simeq 
		\begin{cases}
		\dSchur^{\lambda} \big(R_{\GG_{d+1}^+}^\vee\big)  & \text{if} \quad \lambda \in B_{n-d-1,d} \subseteq B_{n-d,d}; \\
		0 & \text{if} \quad \lambda \in B_{n-d,d} \,\backslash \, B_{n-d-1,d}.
		\end{cases}
		$$
	\item
	\label{lem:flag:Lascoux-2}
	 If $\kk$ is a $\QQ$-algebra, then for any $\lambda \in B_{n-d-1, d}$ with $\lambda_1 = k$, where $\max\{0,d-r+1\} \le k \le d$, the image $\Psi\big(\dSchur^{\lambda} \big(R_{\GG_{d+1}^+}^\vee\big)\big)$ admits a resolution by vector bundles
		$$ \Psi\Big(\dSchur^{\lambda} \big(R_{\GG_{d+1}^+}^\vee\big)\Big) \simeq F_* = \big[0 \to F_k \to \cdots \to F_1 \to F_0\big],$$
	where $F_0 = \dSchur^{\lambda} \big(R_{\GG_{d}^+}^\vee\big)$ and
		$F_i = \dSchur^{\lambda^{(i)}}(R_{\GG_{d}^+}^\vee) \otimes \bigwedge\nolimits^{|\lambda^{(i)}| - |\lambda|}(W)$ 
		for $1 \le i \le k.$ Here, for any given $1 \le i \le k$, let $1 \le j \le n-d-1$ be such that $\lambda_j  \ge i \ge \lambda_{j+1}+1$, then 
	\begin{equation}\label{eqn:lambda^(i)}
		\lambda^{(i)} = (\lambda_1, \lambda_2, \ldots, \lambda_j, i, \lambda_{j+1}+1, \ldots, \lambda_{n-d-1}+1) \in B_{n-d,k} \,\backslash \,B_{n-d-1,k}.
	\end{equation}
\end{enumerate}
\end{lemma}

\begin{proof}
First, we prove assertion \eqref{lem:flag:Lascoux-1}.
Using the Notation(s) \ref{notation:Psi_k} (and \ref{notation:abuse:local}), there is a short exact sequence of vector bundles on $\Flag(\sE;d,d+1)$:
	$$\sL_{d+1}^\vee \hookrightarrow p_+^*\big(R_{\GG_{d}^+}^\vee\big) \twoheadrightarrow p_-^*\big(R_{\GG_{d+1}^+}^\vee \big),$$
where $\sL_{d+1} = \Ker(\sQ_{d+1} \to \sQ_{d})$, and $p_\pm$ are defined as in  \eqref{diag:corr:flagd,d+1}. 
Let $\lambda \in B_{n-d,d}$, then from the from direct-sum decomposition formula \cite[Theorems 2.5 (b)]{Kou} (see also \cite[Theorem 2.12]{J22b}), 
there is filtration on $\dSchur^{\lambda} \Big(p_+^*\big(R_{\GG_{d}^+}^\vee\big) \Big) \simeq p_+^* \Big(\dSchur^{\lambda}  \big(R_{\GG_{d}^+}^\vee\big) \Big)$ whose associated graded is 
	$$\bigoplus_{\nu=(\nu_{1}, \ldots, \nu_{n-d-1}) \subseteq \lambda=(\lambda_{1}, \ldots, \lambda_{n-d}) } \dSchur^{\nu}\Big(p_-^*\big(R_{\GG_{d+1}^+}^\vee \big)\Big) \otimes \dSchur^{\lambda/\nu} (\sL_{d+1}^\vee). \footnote{notice that we switch the roles of $N$ and $L$ in \cite[Theorem 1.4 (b)]{Kou}, and our version follows from \cite[Theorem 2.5 (b)]{Kou} by applying the duality $\dWeyl^{\lambda/\mu}(\sU^\vee)^\vee \simeq \dSchur^{\lambda/\mu}(\sU)$ for vector bundles $\sU$.}  $$
The skew Schur module $\dSchur^{\lambda/\nu} (\sL_{d+1}^\vee)$ is a quotient of tensor products $\bigotimes_{i=1}^{\lambda_1}\bigwedge^{\lambda^t_i - \nu^t_i}(\sL_{d+1}^\vee)$ (see \cite[Notation 2.3]{J22b} or \cite[\S 2.1]{Wey}), where $\lambda^t = (\lambda_1^t, \lambda_2^t, \ldots) $ and $\nu^t = (\nu_1^t, \nu_2^t, \ldots)$ are transposes of $\lambda$ and $\nu$, respectively. As a result, $\dSchur^{\lambda/\nu}(\sL_{d+1}^\vee)$ is zero unless
		\begin{equation}\label{eqn:lambda.nu.condition}
	0 \le \lambda_{n-d} \le \nu_{n-d-1} \le \lambda_{n-d-1} \le \ldots \le \nu_2 \le \lambda_2 \le \nu_1 \le \lambda_1 \le d,
		\end{equation}
in which case the corresponding summand of the associated graded is equvivalent to
	$$\dSchur^{\nu}\Big(p_-^*\big(R_{\GG_{d+1}^+}^\vee \big)\Big) \otimes (\sL_{d+1}^\vee)^{|\lambda| - |\nu|} \simeq 
	p_-^* \Big(\dSchur^{\nu}\big(R_{\GG_{d+1}^+}^\vee \big) \Big) \otimes (\sL_{d+1}^\vee)^{|\lambda| - |\nu|}.$$
	
If $\lambda \in B_{n-d-1,d}$, the partitions $\nu$ appearing in \eqref{eqn:lambda.nu.condition} can be classified into two cases:
\begin{itemize}
	\item Case $\nu=\lambda$: In this case, the corresponding summand is equivalent to $p_-^* \Big(\dSchur^{\nu}\big(R_{\GG_{d+1}^+}^\vee \big) \Big)$.
	\item Case $\nu \neq \lambda$: In this case, we have $1 \le |\lambda| - |\nu| \le \lambda_1 \le d$. For such cases, Serre's vanishing theorem (Remark \ref{rmk:flag:p_pm}, 
\cite[Theorem 5.2]{J22a}) implies that $p_{-\,!} ((\sL_{d+1}^\vee)^{|\lambda| - |\nu|}) \simeq 0$.
\end{itemize}
 Consequently, we obtain that $p_{-\,!} \, p_+^* \Big(\dSchur^{\lambda}  \big(R_{\GG_{d}^+}^\vee\big) \Big) \simeq  \dSchur^{\lambda} \big(R_{\GG_{d+1}^+}^\vee\big)$ as claimed. 

If $\lambda \in B_{n-d,d} \,\backslash\, B_{n-d-1,d}$ which means that  $\lambda_{n-d} \ge 1$ and $\lambda_1 \le d$, then we have $1 \le |\lambda| - |\nu| \le \lambda_1 \le d$ for all partitions $\nu$ satisfying  \eqref{eqn:lambda.nu.condition}. By applying Serre's vanishing theorem once again, we conclude that $p_{-\,!} \, p_+^* \Big(\dSchur^{\lambda}  \big(R_{\GG_{d}^+}^\vee\big) \Big) \simeq  0$ as desired.

\medskip

Next, we prove assertion \eqref{lem:flag:Lascoux-2}. Similarly as with Lemma \ref{lem:flip}, \cite[Proposition 4.19]{J22b} implies that the projection $p_+$ factorizes through a composite map
	$$\Flag(\sE;d,d+1) \xrightarrow{\iota} \PP_{\Grass(\sE;d)} (R_{\GG_d^+}) \xrightarrow{\pr} \Grass(\sE;d)$$
where $\iota$ is a closed immersion induced by a regular section of the vector bundle $W^\vee \otimes \sL_{d+1}$, and $\pr$ is the canonical projection. Therefore, we have a canonical equivalence
	$$\Psi\Big(\dSchur^{\lambda} \big(R_{\GG_{d+1}^+}^\vee\big)\Big)  \simeq \pr_*\Big( \dSchur^{\lambda} \big(R_{\GG_{d+1}^+}^\vee\big)  \otimes \iota_{*} (\sO_{\Flag(\sE;d,d+1)}) \Big),$$
where $\iota_{*} (\sO_{\Flag(\sE;d,d+1)})$ is resolved by a Koszul complex whose $\ell$th terms are given by
	$$\big(\bigwedge\nolimits^{\ell}  W \big) \otimes \sL_{d+1}^{-\ell}  \quad \text{where} \quad 0 \le \ell \le m.$$
Since $\bigwedge\nolimits^{\ell}  W \simeq 0$ if $\ell >m$, we may assume $0 \le \ell \le n+k-d-1$.
By considering the spectral sequence which computes the above higher direct image $\pr_*\Big( \dSchur^{\lambda} \big(R_{\GG_{d+1}^+}^\vee\big)  \otimes \iota_{*} (\sO_{\Flag(\sE;d,d+1)}) \Big)$ (see \cite[Lemma B.1.5]{Laz}), it suffices to compute derived pushforwards of the form
	\begin{equation}\label{eqn:pushforward_dSchur_L^-ell}
	\pr_* \Big(\dSchur^{\lambda} \big(R_{\GG_{d+1}^+}^\vee\big) \otimes  \sL_{d+1}^{-\ell} \Big) \otimes \big(\bigwedge\nolimits^{\ell}  W\big) [\ell] \qquad 0 \le \ell \le n+k-d-1.
	\end{equation}
Using the equivalence $\PP(R_{\GG_d^+}) \simeq \Grass(R_{\GG_d^+}^\vee;n-d-1)$ and Theorem \ref{thm:BWB}.\eqref{thm:BWB-1}, we have
	$$\pr_* \Big(\dSchur^{\lambda} \big(R_{\GG_{d+1}^+}^\vee\big) \otimes  \sL_{d+1}^{-\ell} \Big) \simeq \pi_* \big( \sL(\lambda,\ell)\big),$$
where $\pi \colon \Flag(R_{\GG_d^+}^\vee; \underline{n-d}) \to \Grass(\sE;d)$ is the complete flag bundle of $R_{\GG_d^+}^\vee$ over $\Grass(\sE;d)$, and $\sL(\lambda,\ell)$ is the line bundle associated with the sequence 
	$(\lambda,\ell)= (\lambda_1, \ldots, \lambda_{n-d-1}, \ell).$
According to Borel--Weil--Bott theorem, let $\rho = (n-d-1, n-d-2, \ldots, 2, 1, 0)$, to compute the derived pushforward $\pi_* \big( \sL(\lambda,\ell)\big)$ it suffices to analyze the sequence
	\begin{equation}\label{eqn:widetildelambda+rho}
		(\lambda,\ell) + \rho= (\lambda_1 + n-d-1, \lambda_2 + n-d-2,  \ldots, \lambda_{n-d-1}+1, \ell).
	\end{equation}
		
First, we consider the case $\ell=0$. In this case, \eqref{eqn:widetildelambda+rho} is a partition, and Borel--Weil theorem (see \cite[Theorems 4.1.4]{Wey} or Theorem \ref{thm:BWB}.\eqref{thm:BWB-1}) implies that \eqref{eqn:pushforward_dSchur_L^-ell} is isomorphic to
	$$\pr_* \Big(\dSchur^{\lambda} \big(R_{\GG_{d+1}^+}^\vee\big) \Big)\simeq  \dSchur^{\lambda} \big(R_{\GG_{d}^+}^\vee\big). $$
	
Next, we consider the the case where $1 \le \ell \le n+k-d-1$. From Borel--Weil--Bott theorem (see \cite{Dem}, \cite[Theorem 4.1.10]{Wey} or Theorem \ref{thm:BWB}.\eqref{thm:BWB-2}), 
we obtain that \eqref{eqn:pushforward_dSchur_L^-ell} is nonzero only if the entries of  \eqref{eqn:widetildelambda+rho} are pairwise distinct. There are precisely $(n+k-d-1) - (n-d-1) = k$ such choices for $\ell$, all of the form $\lambda_j + n -d - j > \ell > \lambda_{j+1} + n - d - (j+1)$, where $1 \le j \le n-d-1$. For each such $\ell$ and $j$, it requires a minimal number of $(n-d-1-j)$ permutations of entries of \eqref{eqn:widetildelambda+rho} such that the resulting sequence
	$$(\lambda_1 + n-d-1, \ldots, \lambda_j + n -d - j, ~ \ell , ~ \lambda_{j+1} + n - d - (j+1) , \ldots, \lambda_{n-d-1}+1)$$
is strictly decreasing. Subtracting $\rho$ from the above sequence, we obtain a partition 
	$(\lambda_1, \ldots, \lambda_j, ~ \ell+j-(n-d-1), ~\lambda_{j+1} +1, \ldots, \lambda_{n-d-1}+1),$
which precisely corresponds to the partition $\lambda^{(i)}$ of \eqref{eqn:lambda^(i)}, where $i = \ell+j-(n-d-1)$. 

Conversely, for any $1 \le i \le k$, we let $1 \le j \le n-d-1$ be such that $\lambda_j  \ge i \ge \lambda_{j+1}+1$. In this case, $\ell := |\lambda^{(i)}| - |\lambda|$  is the unique integer in $[1, n+k-d-1]$ such that $\lambda_j + n -d - j > \ell > \lambda_{j+1} + n - d - (j+1)$. For each such $i$ and $\ell$, the Borel--Weil--Bott theorem implies that \eqref{eqn:pushforward_dSchur_L^-ell} is canonically equivalent to
	$$\dSchur^{\lambda^{(i)}}\big(R_{\GG_d^+}^\vee\big)[\ell-(n-d-1-j)] \otimes \big(\bigwedge\nolimits^{\ell}  W\big) = \dSchur^{\lambda^{(i)}}\big(R_{\GG_d^+}^\vee\big) \otimes \big(\bigwedge\nolimits^{|\lambda^{(i)}| - |\lambda|}  W\big) [i].$$
Hence the lemma is proved.
\end{proof}

Notice that Lemma \ref{lem:flag:Lascoux}.\eqref{lem:flag:Lascoux-2} is the only part of the proof of the main theorem in this paper where the characteristic-zero assumption is required.

\begin{corollary}
\label{cor:induction:pattern}
Assume we are in the same situation as Lemma \ref{lem:flag:Lascoux}.\eqref{lem:flag:Lascoux-2} and let $k$ be an integer such that $\max\{0,d-r+1\} \le k \le d$. We let $\Psi_i = \Psi $ be defined as in Notation \ref{notation:Psi_k}; that is $\Psi_i = \Psi(\blank) \otimes \det(\sQ_{\Grass(\sE;d)})^{\otimes i}
 \simeq \Psi(\blank) \otimes \det(R_{\GG_{d}^+}^\vee)^{\otimes i}$. 
For any $\ell, d' \ge 0$, we define 
	$$\shB_{\ell, d'} = \Big\langle \big\{ \dSchur^{\lambda}\big(R_{\GG_{n-\ell}^+}^\vee\big) \}_{\lambda \in B_{\ell,d'}}  \Big \rangle\subseteq \D(\Grass(\sE; n-\ell)).$$ 
Then for each integer $0 \le i \le k - \max\{0,d-r+1\}$, the restriction of the functor $\Psi_i$,
	$$\Psi_i|_{\shB_{n-d-1,k-i}} \colon  \shB_{n-d-1,k-i} \to \D(\Grass(\sE;d)),$$ 
is fully faithful, with essential image contained in $\shB_{n-d,k}$. Moreover, these functors $\Psi_i|_{\shB_{n-d-1,k-i}}$, for $0 \le i \le k - \max\{0,d-r+1\}$, induce a semiorthogonal decomposition 
\begin{equation}\label{eqn:sod:induction}
\shB_{n-d,k} =  
	\Big \langle  
\big \langle \Psi_{k-i} (\shB_{n-d-1,i}) \big\rangle_{i \in [0,  \max\{0,d-r+1\}]} ~, ~
\Psi_{k-d+r}^0(\shB_{n-d,d-r} ) \Big \rangle,
\end{equation}
where $\Psi_{k-d+r}^{0}$ denotes the functor $\otimes  \det(\sQ_{\Grass(\sE;d)})^{\otimes (k-d+r)}$, the last component is understood as empty if $d < r$, and the semiorthogonal order of the first part is given by the usual order $<$ of integers in $[0,  \max\{0,d-r+1\}]$, that is: for all $0 \le j < i \le  k - \max\{0,d-r+1\}$, $\Map( \Psi_{k-i} (\shB_{n-d-1,i}),  \Psi_{k-j} (\shB_{n-d-1,j})) \simeq 0$.

\begin{proof}
We will only prove the case where $d \ge r$; the other case where $d < r$ is similar and simpler. Notice that Lemma \ref{lem:flag:Lascoux}.\eqref{lem:flag:Lascoux-2} implies that $\Psi_i\big(\dSchur^{\lambda}(R_{\GG_{d+1}^+}^\vee)\big)  \in \big \langle \{\dSchur^{\lambda}(R_{\GG_{d}^+}^\vee)\}_{\lambda \in B_{n-d,k}} \big \rangle$ for all $0 \le i \le k - d+r-1$ and $\lambda \in B_{n-d-1,k-i}$. This proves the assertion that the essential image $\Psi_i|_{\shB_{n-d-1,k-i}}$ is contained in $\shB_{n-d,k}$. 

To establish the desired semiorthogonal decomposition \eqref{eqn:sod:induction}, it suffices to prove:
\begin{enumerate}[leftmargin=*]

	\item[$(a)$] Fully-faithfulness: The counit map $\Psi_i^L \Psi_i \to \id$ is an equivalence when restricted to the subcategory $\shB_{n-d-1,k-i}$, where $0 \le i \le k - d+r-1$. As with Corollary \ref{cor:ff:incidence} and from the definition of $\shB_{n-d-1,k-i}$, it suffices to prove that 
	 $\Psi_i^L \Psi_i \Big(\dSchur^{\lambda}\big(R_{\GG_{d+1}^+}^\vee\big)\Big) \to \dSchur^{\lambda}(R_{\GG_{d+1}^+}^\vee)$ is an equivalence for all $\lambda \in B_{n-d-1,k-i}$. This follows from Lemma \ref{lem:flag:Lascoux}.\eqref{lem:flag:Lascoux-1}-\eqref{lem:flag:Lascoux-2}.
	 
	 \item[$(b)$] Semiorthogonality:
	 \begin{enumerate}[leftmargin=*]
		\item 
		For all $0 \le j < i \le k - d+r-1$, $\Psi_j^L \Psi_i \simeq 0$ when restricted to the subcategory $\shB_{n-d-1,k-i}$. As before, it suffices to prove that $\Psi_j^L \Psi_i \Big(\dSchur^{\lambda}\big(R_{\GG_{d+1}^+}^\vee\big)\Big)\simeq 0$ for all $\lambda \in B_{n-d-1,k-i}$. This is again a direct consequence of Lemma \ref{lem:flag:Lascoux}.\eqref{lem:flag:Lascoux-1}-\eqref{lem:flag:Lascoux-2}.
		\item
		For all $0 \le j < i \le k - d+r-1$, the restriction of the functor $\Psi_i^L$ to $\Psi_{k-d+r}^0(\shB_{n-d,d-r})$ is equivalent to zero. Once again, it suffices to prove that for any $\alpha \in B_{n-d,d-r}$,  
		$\Psi_i^L \Big(\dSchur^{\alpha}\big(R_{\GG_{d}^+}^\vee\big) \otimes \det(R_{\GG_{d}^+}^\vee)^{(k-d+r)} \Big)\simeq 0$.
		Since $d \ge r$ and $0 \le i \le k-d+r-1$, we have  $1 \le k-d+r -i \le d$. Hence $\dSchur^{\alpha}\big(R_{\GG_{d}^+}^\vee\big) \otimes \det(R_{\GG_{d}^+}^\vee)^{(k-d+r)} \in B_{n-d,d} \backslash B_{n-d-1,d}$, and the desired result follows from Lemma \ref{lem:flag:Lascoux}.\eqref{lem:flag:Lascoux-1}. 
	\end{enumerate}
	
	\item[$(c)$] Generation: 
	To complete the proof, we will show that any element $\dSchur^{\alpha}(R_{\GG_{d}^+}^\vee)$, where $\alpha \in B_{n-d,d-r}$, belongs to the right-hand side of \eqref{eqn:sod:induction}. 
	We will establish this result using induction.
	Let us introduce the following notations:
	for any $\nu \in B_{n-d,w}$ and $i \in \ZZ$, where $w \ge 0$ is an integer, we let $\nu(i) = (\nu_1+i, \ldots, \nu_{n-d}+i)$. Let $B_{n-d,w}(i) : = \{\nu(i) \mid \nu \in B_{n-d,w}\}$. 
	Using these notations, we can express a disjoint union decomposition as follows:
	$$B_{n-d,k} = B_{n-d,d-r}(k-d+r) \sqcup \bigsqcup_{i=0}^{n-d+r-1} B_{n-d-1,k-i}(i).$$
It is clear that if $\alpha \in B_{n-d,d-r}(k-d+r)$, meaning that $\alpha = \nu(k-d+r)$ for some $\nu \in B_{n-d,d-r}$, then $\dSchur^{\alpha}(R_{\GG_{d}^+}^\vee) = \dSchur^{\nu}(R_{\GG_{d}^+}^\vee) \otimes \det(R_{\GG_{d}^+}^\vee)^{(k-d+r)}$ belongs to the right-hand side of \eqref{eqn:sod:induction}. 
Now we assume that $\alpha \in B_{n-d-1,k-i}(i)$, that is, $\alpha = \nu(i)$ for some $\nu \in B_{n-d-1,k-i}$. According to Lemma \ref{lem:flag:Lascoux}.\eqref{lem:flag:Lascoux-2}, there is a canonical map $\dSchur^{\alpha}(R_{\GG_{d}^+}^\vee)  \to \Psi_i\big(\dSchur^{\nu}(R_{\GG_{d+1}^+}^\vee)\big)$, and the cone of this map is given by iterated extensions of elements of the form $\dSchur^{\beta}(R_{\GG_{d}^+}^\vee) \otimes K$, where $\beta \in B_{n-d,d-r}(k-d+r) \sqcup \bigsqcup_{j=i+1}^{n-d+r-1} B_{n-d-1,k-j}(j)$ and $K$ is a finite free $\kk$-module. Consequently, the desired result regarding generation follows from induction.
\end{enumerate}
\end{proof}
\end{corollary}

\subsubsection{Proof of Theorem \ref{thm:SOD}, Part 2}
\label{subsec:proof:SOD}
We now complete the proof of Theorem \ref{thm:SOD} by establishing the theorem in the universal local situation  \S \ref{subsec:setup:local} and when $\D = \Dperf$, using the preparations made in the preceding subsections.

If $d = 0$ or $r=0$, the desired result follows directly from Corollary \ref{cor:ff:incidence}. Therefore, we may assume $d$ and $r$ are both greater than zero. 

We now generate a semiorthogonal decomposition of $\D(\Grass(\sE;d))$ by iteratively applying Corollary \ref{cor:induction:pattern}. Let us describe the process:

\begin{itemize}
	\item[(*)]
	Starting with the case where $k=d$, we apply Corollary \ref{cor:induction:pattern} and obtain a semiorthogonal decomposition of $\D(\Grass(\sE;d)) = \shB_{n-d,d}$. This decomposition takes the form of the form \eqref{eqn:sod:induction} and its components are given by the images $\Psi_i (\shB_{n-d-1, d-i})$ for $0 \le i \le d - \max\{0, d-r+1\}$, and $\Psi_{d-r}^0(\shB_{n-d,d-r})$ (if $d \ge r$). Notably, the appearing subcategories $\shB_{a,b}$ as the domains of $\Psi_i$ or $\Psi_{d-r}^0$ satisfy the condition $n - r \le a+b \le n - 1$. 

	For any subcategory $\shB_{a,b}$ appearing in the above decomposition with $a+b>n -r$ (implying $\shB_{a,b} = \shB_{n-d-1, d-j}$ for some $j \ge 0$), we further decompose $\shB_{a,b}$ by applying Corollary \ref{cor:induction:pattern} again. The involved subcategories $\shB_{a',b'}$ appearing of this decomposition  again satisfy  $n - r \le a'+b' \le a+b-1 \le n - 2$. We continue this process, applying Corollary \ref{cor:induction:pattern} to each involved subcategory $\shB_{a',b'}$ such that $a'+b'> n-r$, until all the subcategories are of the form $\shB_{a'',b''}$, where $a''+b''=n-r$.
\end{itemize}

The above process $(*)$  clearly terminates in a finite number of steps. At the end, we obtain a semiorthogonal decomposition of $\D(\Grass(\sE;d))$ whose components are given by fully faithful images of subcategories of the form $\shB_{n-d-r+i,d-i}$, where $0 \le i \le \min\{r,d\}$. Each such category $\shB_{n-d-r+i,d-i}$ is embedded via a functor of the form:
\begin{equation}\label{eqn:thm:SOD:local:embedding}
	\Psi_{a_1} \circ \Psi_{a_2} \circ \cdots \circ \Psi_{a_{r-i}} \circ \Psi_{i - \sum a_j}^0 \colon \shB_{n-d-r+i,d-i} \to \D(\Grass(\sE;d)),
	\end{equation}
where $\Psi_{i - \sum a_j}^0= \otimes \det(\sQ_{\Grass(\sE;d+r-i)})^{i - \sum a_j}$; the notation $\Psi_{i - \sum a_j}^0$ indicates that it is a ``zero-times composition of $\Psi$'s, further twisted by a line bundle of degree $(i - \sum a_j)$". Here, $a_1, \ldots, a_{r-i} \ge 0$ is a (possibly empty) sequence of integers with with $\sum a_j \le i$. If $i=r$, we understand $a_1,\ldots,a_0$ as the empty sequence and \eqref{eqn:thm:SOD:local:embedding} as the functor $\Psi_{r}^0 = \otimes \det(\sQ_{\Grass(\sE;d)})^{r}$.

Conversely, for any (possibly empty) sequence of integers $a_1, \ldots, a_{r-i} \ge 0$ with $\sum a_j \le i$, there is precisely one copy of $\shB_{n-d-r+i,d-i}$ embedded as the image of the functor \eqref{eqn:thm:SOD:local:embedding} in the semiorthogonal decomposition obtained through the above process $(*)$. Moreover, for any given $0 \le i \le \min\{r,d\}$, such a (possibly empty) sequence $a_1, \ldots, a_{r-i}$ is in one-to-one correspondence with a (possibly zero) partition $\lambda \in B_{r-i,i}$ via the formula
	\begin{align}
	&a_1  = i -\lambda_1, \quad   a_2 = \lambda_1 - \lambda_2, \quad  \ldots, \quad  a_{r-i} = \lambda_{r-i-1} - \lambda_{r-i}. \label{eqn:SOD:a_lambda}
	\end{align}
Here, if $r=i$, the empty sequence $a_1, \ldots, a_0$ corresponds to the zero partition $(0)\in B_{0,r}$.

For each such (possibly empty) sequence $a_1, \ldots, a_{r-i}$ (or equivalently, for each partition $\lambda \in B_{r-i,i}$, in view of \eqref{eqn:SOD:a_lambda}), composing \eqref{eqn:thm:SOD:local:embedding} with the equivalence of Corollary \ref{cor:ff:incidence},
	$$\Phi^{(0)}_{(d+r-i,d-i)} \colon \D(\Grass(\sE^\vee[1];d-i)) \xrightarrow{\simeq} \shB_{n-d-r+i,d-i},$$
we obtain precisely one copy of $\D(\Grass(\sE^\vee[1];d-i))$ embedded into $\D(\Grass(\sE;d))$ via the fully faithful functor
	$$\Psi_{a_1} \circ \cdots \circ \Psi_{a_{r-i}} \circ (\otimes \det(\sQ_{\Grass(\sE;d+r-i)})^{i - \sum a_j})  \circ  \Phi_{(d+r-i,d-i)}^{(0)} \simeq \Phi^{(i, \lambda)}_{(d,d-i)},$$
where the last equivalence follows from Corollary \ref{cor:forg:FM:Incidence} and \eqref{eqn:SOD:a_lambda}.

	To summarize, for each $0 \le i \le \min\{r,d\}$ and each partition $\lambda \in B_{r-i,i}$, we obtain an embedding of $\D(\Grass(\sE^\vee[1], d-i))$ into $\D(\Grass(\sE,d))$ via the fully faithful functor $\Phi^{(i,\lambda)} = \Phi^{(i, \lambda)}_{(d,d-i)}$. All the components produced in the process $(*)$ can be expressed in this form in a unique way. Therefore, we have obtained the desired semiorthogonal decomposition.

Furthermore, it is clear from the Corollary \ref{cor:induction:pattern} and the process $(*)$ that the resulting semiorthogonal decomposition has the semiorthogonal order given by the lexicographic order $<_{\rm lex}$ of the sequences $(a_1, a_2, \ldots, a_{r-i}, 0, 0, \ldots)$ indexing the components embedded via the functors \eqref{eqn:thm:SOD:local:embedding}, where the empty sequence represents the largest element. 
In view of \eqref{eqn:SOD:a_lambda}, this is equivalently to the order $<_{\rm diff}$ on the pairs $(i,\lambda)$ defined in Notation \ref{notation:differene.order}.

 This concludes the proof of Theorem \ref{thm:SOD} in the universal local situation. By combining it with the argument presented in \S \ref{sec:SOD}, we have completed the proof of Theorem \ref{thm:SOD}. \hfill $\square$


\section{Applications}
In this section, we explore some of the applications of Theorem \ref{thm:SOD} in  classical scenarios.  
We will fix a $\QQ$-algebra $\kk$, and consider schemes, morphisms, and classical fiber products within the category of $\kk$-schemes. We let the symbol $\D$ denote $\Dqc$, $\Dpc$, $\Dltpc$ or $\Dperf$.

\subsection{Blowups of Determinantal Ideals}
\label{sec:blowup.det}
We consider a  $\kk$-scheme with $Z \subseteq X$ a determinantal subscheme of codmension $(r +1)$, where $r \ge 1$. For simplicity, we define $Z$ as the zero subscheme of a Fitting ideal ${\rm Fitt}_{r}(\sH^0(\sE))$ (see \cite[\href{https://stacks.math.columbia.edu/tag/0C3C}{Tag 0C3C}]{stacks-project}), where $\sE$ is a perfect complex with Tor-amplitude in $[0,1]$ and rank $r$ and $\sH^0(\sE)$ is the zeroth sheaf homology of $\sE$.

We consider the projection $\pi =\pr_{\Grass(\sE;r)} \colon \Grass_X(\sE;r) \to X$.
Assuming that $\Grass_X(\sE;r)$ is a classical scheme and that $E:=\pi^{-1}(Z) \subseteq \Grass_X(\sE;r)$ is an effective Cartier divisor, then $\Grass_X(\sE;d)$ is isomorphic to the (classical) blowup 
	$\pi \colon \Bl_{Z}(X) = \underline{\Proj}_{X} \bigoplus_{n \ge 0} \sI_Z^n \to X$ 
	of $X$ along $Z$, with $\det (\sQ_{\Grass(\sE;d)}) = \sO_{\Bl_Z X}(-E) \otimes \det (\sE)$ (see \cite[Lemma 2.24]{J21}).  

For each $j \ge 0$, we let $X_j$ ($=X^{\ge r+j}(\sH^0(\sE))$ of \cite{J21}) be the closed subscheme defined by the Fitting ideal ${\rm Fitt}_{r-1+j}(\sH^0(\sE))$; notice then $X_0 = X$ and $X_1 = Z$. We write
	$$\widetilde{X_j} : = \Grass_X(\sE^\vee[1]; j) \to X.$$
The underlying classical map of $\widetilde{X_j} \to X$ factorizes through $X_j$, and is an isomorphism over $X_j \, \backslash X_{j+1}$ (see \cite[\href{https://stacks.math.columbia.edu/tag/05P8}{Tag 05P8}]{stacks-project} or \cite[Corollary 2.8]{J21}). Therefore, we can view $\widetilde{X_j}$ as a {\em (possibly derived) partial desingularitzation} of the higher determinantal subscheme $X_j$. For example, in the case where $X$ is an irreducible Cohen--Macaulay subscheme and $X_j \subseteq X$ have expected codimensions $j(r+j)$ for all $j \ge 1$, then $\widetilde{X_j}$ are classical irreducible Cohen--Macaulay schemes and $\widetilde{X_j} \to X_j$ are ${\rm IH}$-small partial desingularitzations (see \cite[Theorem 5.2]{J20}) for all $j \ge 1$.
	
For each $j \ge 0$, the incidence locus $\Incidence_{(r,j)}(\sE)$ is a possibly derived scheme, whose underlying classical scheme is the classical fiber product $\Bl_{Z}(X)\times_{X}^{\rm cl} \widetilde{X_j}$, and $\sE_{(r,j)}^{\rm univ}$ is a universal perfect complex of rank $j$ and Tor-amplitude in $[0,1]$ on $\Incidence_{(r,j)}(\sE)$. 
For each $j \ge 0$ and $\lambda \in B_{j, r-j}$, we consider the Fourier--Mukai functors
	$$\Omega_{(j,\lambda)} :=  r_{+\,*} \left(r_-^*(\blank) \otimes \dSchur^{\lambda}(\sE^{\rm univ}_{(r,j)}) \right) \otimes \sO_{\Bl_Z(X)}(j E)  \colon \D(\widetilde{X_j}) \to \D(\Bl_Z(X)),$$
where $r_\pm$ are the natural projection maps in the incidence diagram \eqref{diagram:Inc:Corr},
	$$\widetilde{X_j} \xleftarrow{r_-} \Incidence_{(r,j)}(\sE) \xrightarrow{r_+} \Bl_Z(X).$$
We denote the essential image 	of $\Omega_{(j,\lambda)}$ by $\D(\widetilde{X_j})_{(j,\lambda)}$.
When $j=0$, the functor $\Omega_{(0,(0))}$ is the pullback functor $\pi^* \colon \D(X) \to \D(\Bl_Z(X))$, and we denote its essential image by $\D(X)_0$. 

As a result of our main theorem  \ref{thm:SOD}, we obtain the following corollary:

\begin{corollary}[Blowup formula for determinantal ideals]
\label{cor:SOD:blowup.det}
In the situation of a determinantal subscheme $Z \subseteq X$ of codmension $(r +1)$ as described above, where $r \ge 1$, the functors $\Omega_{(j,\lambda)}$ are fully faithful for all $j \ge 0$ and $\lambda \in B_{j, r-j}$. Moreover, these functors $\Omega_{(j,\lambda)}$, where $0 \le j \le r$ and $\lambda \in B_{j, r-j}$, induce a semiorthogonal decomposition 
\begin{align*}
 \D\left(\Bl_{Z}(X)\right) = 
   \left \langle  
	\Big \langle  \D(\widetilde{X_j})_{(j,\lambda)}
	\mid 1 \le j \le r, \, \lambda \in B_{j,r-j}
	\Big \rangle,
	  ~ \D(X)_0
  \right \rangle,
\end{align*}
with semiorthogonal order given as follows:  $\Map\big(  \D(\widetilde{X_j})_{(j,\lambda)},  \D(\widetilde{X_k})_{(k,\mu)} \big) \simeq 0$ if $(r-k,\mu) < (r-j,\lambda)$, where $<$ is the total order defined in Notation \ref{notation:differene.order}.
\end{corollary}

This result generalizes both Orlov's blowup formula \cite{Orlov92} and the formula for blowups of Cohen--Macaulay subschemes of codimension $2$ (\cite{JL18, J21}).
If we base-change the above semiorthogonal decomposition to the Zariski open subset $X \backslash Z_2$, we recover Orlov's blowup formula for the local complete intersection (l.c.i.) closed immersion $(Z\backslash Z_2) \subseteq (X \backslash Z_2)$. Therefore, the above formula extends Orlov's to the non-l.c.i. loci of $Z \subseteq X$. The ``corrections" to Orlov's formula in this situation are precisely given by copies of derived categories of the partial resolutions $\widetilde{Z_j}$ of the higher determinantal loci $Z_j \subseteq Z$ for $2 \le j \le r$.

\begin{remark}
Even if we don't assume that $\Grass_X(\sE;r)$ is classical and $\pi^{-1}(Z) \subseteq \Grass_X(\sE;r)$ is an effective Cartier divisor, the semiorthogonal decomposition described in Corollary \ref{cor:SOD:blowup.det} still applies to $\D(\Grass_X(\sE;r))$. However, in this situation, $\Grass_X(\sE;r)$ is no longer isomorphic to the classical blowup $\Bl_Z(X)$. Instead, we should regard $\Grass_X(\sE;r)$ as a derived version of blowup of $X$ along $Z$. We expect this perspective to be closely related to the concept of a derived blowup of Hekking, Khan and Rydh (see \cite{KR18, He21}).
\end{remark}

\subsection{Reducible Schemes}
\label{sec:reducible}
We consider two classes of reducible schemes. 
\subsubsection{}
Let $X$ be a $\kk$-scheme, and let $Z \subseteq X$ a regularly immersed closed subscheme of codmension $r \ge 1$ with normal bundle $\sN_{Z/X}$. For simplicity, we assume that $Z$ is the zero locus of a regular section $s$ of a rank $r$ vector bundle $\sV$ over $X$. 
We also consider a line bundle $\sL$ on $X$, and denote by $\sL_Z$ the restriction of $\sL$ to $Z$. We define a perfect complex $\sE$ of Tor-amplitude in $[0,1]$ and rank $r$ as follows:
	$$\sE = \left [ \sO_X \xrightarrow{(s, 0)^T} \sV \oplus \sL \right ] \quad \text{so that} \quad \sE^\vee[1] \simeq   \left [ \sV^\vee \oplus \sL^\vee \xrightarrow{(s^\vee, 0)} \sO_X \right]. $$
We have the following observations:
\begin{enumerate}[leftmargin=*]
	\item The derived Grassmannian $\Grass_X(\sE; r)$ is isomorphic to the classical reducible scheme
		$$\Bl_Z(X) \bigsqcup_{\PP_{Z}(\sN_{Z/X}^\vee)} \PP_{Z}(\sN_{Z/X}^\vee \oplus \sL_Z^\vee),$$ 
where $\PP_{Z}(\sN_{Z/X}^\vee) \subseteq \Bl_Z(X)$ is the inclusion of the exceptional divisor, and $\PP_{Z}(\sN_{Z/X}^\vee) \subseteq \PP_{Z}(\sN_{Z/X}^\vee \oplus \sL_Z^\vee)$ is the closed immersion induced by $\sN_{Z/X} \subseteq \sN_{Z/X}  \oplus \sL_Z$. The scheme structure is described as follows. By working Zariski locally on $X$, we may assume that $X = \Spec R$ for some commutative ring $R$, $\sV = \sO_X^r$ and $\sL = \sO_X$, and $s$ is given by a regular sequence $(x_1, \ldots, x_r)$ of $R$. Using \cite[Proposition 4.19]{J22b} and the fact that $s$ is regular, we obtain that the regular closed immersion
	$$\Grass_X(\sE;r) \hookrightarrow \Grass_X(\sO_X \oplus \sO_X^{r}; r) \simeq \Spec R \times \PP^r$$
is identified with the inclusion of the {\em classical} subscheme defined by the equations  
		$$x_i X_j - x_j X_i = 0 \quad \text{for} \quad 1 \le i < j \le r,  \quad \text{and} \quad x_k X_0 =0 \quad \text{for} \quad 1 \le k \le r,$$
	where $[X_0: X_1: \ldots: X_r]$ denotes the homogeneous coordinates of $\PP^r$. 
	
	(In fact, one can work over affine charts of $\PP^r$ as follows. For any $1 \le i \le r$, let $U_i = \{X_i \ne 0\} \simeq \AA^r \subseteq \PP^r$, with affine coordinates $(u_0, \ldots, \widehat{u_i}, \ldots, u_r)$, $u_j = X_j/X_i$ for $j\neq i$. In the local chart $\Spec R \times U_i$,  $\Grass_X(\sE;r)$ is defined by the equations $\{x_j = x_i u_j \mid j\neq 0,r\}$ together with $u_0 \cdot x_i = 0$. The first $(r-1)$ equations $\{x_j = x_i u_j\}$ are precisely the defining equations for the blowup $\Bl_Z(X)$ in $X \times \AA^{r-1} = \Spec R[u_1,\ldots, \widehat{u_i}, \ldots, u_r]$, which we shall denote as $\Bl_Z(X)_{U_i}$. The last equation $u_0 \cdot x_i =0$ defines a normal crossing divisor in $\Bl_Z(X)_{U_i} \times \Spec \kk[u_0]$, where the two divisors $\{u_0 = 0\} \simeq \Bl_Z(X)_{U_i}$ and $\{x_i =0\} \simeq Z \times U_i$ intersect along $\{u_0 = x_i = 0\} \simeq Z \times \{u_0 = 0\} \simeq Z \times \AA^{r-1}$.)
	
\item 
By virtue of \cite[Example 4.35]{J22a}, we have a canonical equivalence
	$$q_Z \colon \PP_X(\sE^\vee[1]) \simeq {\rm Tot}_Z(\sL_Z[-1]) \to Z,$$
where ${\rm Tot}_Z(\sL_Z[-1]) = \Spec \Sym_Z^*(\sL_Z^\vee[1])$ denotes total space of $\sL_Z[-1]$.

\item
 The map $r_-$ exhibits the incidence locus $\Incidence_{(r,1)}(\sE)$ as the projective bundle
		$$r_- =q  \colon \Incidence_{(r,1)}(\sE) \simeq \PP_{{\rm Tot}_Z(\sL_Z[-1])}(p_Z^*(\sN_{Z/X}^\vee \oplus \sL_Z^\vee)) \to {\rm Tot}_Z(\sL_Z[-1]),$$
and the universal perfect complex $\sE_{(r,1)}^{\rm univ}$ is isomorphic to $\sO_{q}(-1)$ (see Lemma \ref{lem:dGrass:incidence}.\eqref{lem:dGrass:incidence-2}). 
	The map $r_+ = \iota$ is a closed immersion (see Lemma \ref{lem:dGrass:incidence}.\eqref{lem:dGrass:incidence-1}). 
	For $1 \le j \le r$, we let 
		$$\Omega_j = \iota_*\big(q^*(\blank) \otimes  \sO_{q}(-j) \big) \colon \D({\rm Tot}_{Z}(\sL_Z[-1]) ) \to \D(\Grass_X(\sE;d)).$$
\end{enumerate}

Therefore, in the above situation, Theorem \ref{thm:SOD} implies that:

\begin{corollary}
\label{cor:SOD:reducible}
The pullback functors $\pr^*_{\Grass(\sE;d)}$ and $\Omega_j$ (where $1 \le j \le r$) are fully faithful. 
Denoting the essential image of $\pr_{\Grass(\sE;d)}^*$ as $\D(X)_0$ and $\Omega_j$ as $\D({\rm Tot}_{Z}(\sL[-1]) )_{-j}$ (where $1 \le j \le r$), we have a semiorthogonal decomposition:
\begin{multline*}
	\D\left(\Bl_Z(X) \bigsqcup_{\PP_{Z}(\sN_{Z/X}^\vee)} \PP_{Z}(\sN_{Z/X}^\vee \oplus \sL_Z^\vee)\right)  \\
	  = \Big\langle \D({\rm Tot}_{Z}(\sL_Z[-1]) )_{-r},  \cdots, \D({\rm Tot}_{Z}(\sL_Z[-1]))_{-1}, ~\D(X)_0 \Big \rangle.
	\end{multline*}
\end{corollary}

\begin{remark}
\label{rmk:deftonormalcone}
In the special case where $\sL = \sO_X$ is trivial, ${\rm Tot}_{Z}(\sL_Z[-1]) = Z[\varepsilon_1]$, where $Z[\varepsilon_1]$ denotes $Z \times \Spec(\Sym^*(\kk[1]))$, and the above semiorthogonal decomposition reduces to
	\begin{align*}
	\D\left(\Bl_Z(X) \bigsqcup_{\PP_{Z}(\sN_{Z/X}^\vee)} \PP_{Z}(\sN_{Z/X}^\vee \oplus \sO_Z)\right)   = \Big\langle \D(Z[\varepsilon_1])_{-r},  \cdots, \D(Z[\varepsilon_1])_{-1}, ~\D(X)_0 \Big \rangle.
	\end{align*}
Here, the scheme appearing on the left-hand side is precisely the central fiber $\rho^{-1}(0)$ in the deformation-to-normal-cone construction, where $\rho$ denotes the natural projection $\Bl_{Z \times \{0\}} (X \times \AA^1) \to \AA^1$. In this case, the above semiorthogonal decomposition agrees with the derived base-change of Orlov's blowup formula \cite{Orlov92} for $\Bl_{Z \times \{0\}} (X \times \AA^1)$ to the central fiber.
\end{remark}

In the special case where $Z=D$ is an effective divisor and $\sL = \sN_{D/X}$, we recover the the semiorthogonal decomposition
\begin{align*}
	\D\big(X \bigsqcup\nolimits_{D} \PP_D^1\big)   = \big\langle\D({\rm Tot}(\shN_{D/X}[-1]) ),  ~\D(X) \big \rangle	
	\end{align*}
of  \cite[Examples 7.21]{J22a}. If we furthermore assume that $X=C$ is a complex curve and $Z = \{p\}$ is a non-singular closed point, we recover the semiorthogonal decomposition
	\begin{align*}
	\D\big(C \bigsqcup\nolimits_{p} \PP^1\big)   = \big\langle \D(\Spec \CC[\varepsilon_1]), \D(C) \big \rangle
	\end{align*}
of \cite[Examples 7.22]{J22a} (see also \cite[Proposition 6.15]{KS}), where $\CC[\varepsilon_1] = \Sym^*_\CC(\CC[1])$ is the ring of derived dual numbers. Hence the above result is a higher-codimensional generalization of the formula \cite[Examples 7.21]{J22a} for attaching $\PP^1$-bundles to divisors.

\subsection{Varieties of Linear Series on Curves}
In this subsection, we consider the case where $\kk = \CC$, and study a family of smooth complex projective curves  $\sC/S$  of genus $g \ge 1$. For simplicity, we assume the existence of a section $\sigma \colon S \to \sC$ of $\sC/S$. 

We denote the classical (rigidified) relative Picard functor of degree $d$ by $\underline{\Pic}_{\sC/S}^{d}$, which assigns to each $S$-scheme $T$ the isomorphism class of pairs $(\sL_T, i)$, where $\sL_T$ is a line bundle on $X_T$ with fiberwise degree $d$, and $i$ is an isomorphism $\sigma^*(\sL_T) \xrightarrow{\simeq} \sO_T$. 
Under this assumption, the functor $\underline{\Pic}_{\sC/S}^{d}$ is representable by a locally projective, smooth $S$-scheme $\Pic_{\sC/S}^d \to S$ of relative dimension $g$ (see \cite{Gro,Kl05}). 

Let $\sL_{\rm univ}$ be the Poincar\'e line bundle on $\Pic_{\sC/S}^d \times_S \sC$ and $\pr \colon \Pic_{\sC/S}^d \times_S \sC \to \Pic_{\sC/S}^d$ the natural projection. 
By applying the argument of \cite[\S 3.1.3]{JL18}, we obtain that 
	$\sE := \left(\pr_* (\sL_{\rm univ}) \right)^\vee$
is a perfect complex on $\Pic_{\sC/S}^d$ of Tor-amplitude in $[0,1]$ and rank $(1-g+d)$. For an integer $r \ge -1$, we define a (possibly derived) scheme
	$$\bG_d^r(\sC/S):  = \Grass_{\Pic_{\sC/S}^d}(\sE; r+1) \to \Pic^d_{\sC/S}.$$
This derived scheme is proper and quasi-smooth over $\Pic^d_{\sC/S}$, and its underlying closed points over a point $s \in S(\CC)$ correspond to the $\CC$-points of the variety $G_d^r(\sC_s)$ of linear series $g_d^r$ of degree $d$ and dimension $r$ on $\sC_s$ as studied in \cite[Chapter IV]{ACGH}.
More specifically, the closed points of $\bG_d^r(\sC/S)$ over $s \in S(\CC)$  are given by the isomorphism classes of the pair $(\sL_s, g_d^r)$, where $\sL_s$ is a line bundle on $\sC_s$ of degree $d$, and $g_d^r$ is a $r$-dimensional linear projective subspace of $\PP^{\rm sub}(\H^0(\sC_s; \sL_s))$. For any $0 \le i \le r+1$, the relative Serre duality implies the isomorphism
	$$\bG_{2g-2-d}^{r-i}(\sC/S)  \simeq \Grass_{\Pic_{\sC/S}^d}(\sE^\vee[1]; r+1-i) \to \Pic^d_{\sC/S}$$	
 whose underlying map carries a pair $(\sL_s, g_{2g-2-d}^{r-i})$ to the line bundle $\sL_s^\vee \otimes \omega_{\sC_s} \in \Pic^{2g-2-d}(\sC_s)$.
 
In this case, Theorem \ref{thm:SOD} yields the following corollary:
\begin{corollary}
\label{cor:SOD:curves}
In the above situation, assuming that $d \geq g-1$ and $r \geq -1$, there exists a semiorthogonal decomposition:
\begin{align*}
\D(\bG_{d}^{r}(\sC/S)) = \left\langle
\text{$\binom{1-g+d}{i}$ copies of } \D(\bG_{2g-2-d}^{r-i}(\sC/S)) \right\rangle_{0 \leq i \leq \min\{1-g+d, r+1\}},
\end{align*}
where the Fourier-Mukai functors and semiorthogonal orders are given as in Theorem \ref{thm:SOD}.
\end{corollary}

Now we focus on the case where $S = \Spec \CC$. We denote $\bG_{d}^{r}(C) = \bG_{d}^{r}(C/\Spec \CC)$. 

If $C$ is a \emph{general} curve, $\bG_d^r(C) = G_d^r(C)$ is the classical variety of linear series on $C$ of degree $d$ and dimension $r$ studied in \cite{ACGH}. Similarly, $\bG_{2g-2-d}^{r-i}(C) = G_{2g-2-d}^{r-i}(C)$ for $0 \le i \le r+1$. Moreover, these varieties are reduced, smooth, and have expected dimensions (\cite[Theorems V.(1.5), V.(1.6)]{ACGH}). They are non-empty precisely when their expected dimensions are non-negative (\cite[Theorems V.(1.1), V.(1.5)]{ACGH}). In this case, Corollary \ref{cor:SOD:curves} implies
	\begin{align*}
	\D(G_{d}^{r}(C)) = \left\langle
\text{$\binom{1-g+d}{i}$ copies of } \D(G_{2g-2-d}^{r-i}(C)) \right\rangle_{0 \leq i \leq \min\{1-g+d, r+1\}}.
	\end{align*}
Additionally, when $C$ is general, it can be shown, following a similar argument as in \cite[Lemmas B.3, B.4]{JL18}, that the incidence schemes $\Incidence_{r+1,r+1-i}(\sE)$ are isomorphic to the classical fiber products $G_d^r(C) \times_{\Pic_C^d} G_{2g-2-d}^{r-i}(C)$ and have expected dimensions. Furthermore, for a general curve $C$, Lin and Yu \cite{LY21} showed that the derived categories $\D(G_{2g-2-d}^{r-i}(C))$ are indecomposable for all $0 \le i \le r+1$.

If $C$ is special, then Corollary \ref{cor:SOD:curves} still holds and reveals many intriguing phenomena. Here, we only focus on two $3$-fold examples:

\begin{example} 
\label{eg:curves:g=5}
If $C$ is a (non-hyperelliptic) trigonal curve of genus $5$, then $W^2_5(C) \simeq W^1_3(C)$ consists of a single point. In this case, $\GG^1_5(C) =G^1_5$ is a classical irreducible singular threefold, and $\GG^0_3 \simeq C^{(3)}$ is classical and smooth. Corollary \ref{cor:SOD:curves} implies
	$$\D(G^1_5(C)) = \big\langle \D(\bG^1_3(C)), ~ \D(C^{(3)}) \big\rangle,$$
where $\bG^1_3(C)$ is a nonclassical derived scheme with virtue dimension $-1$, and and has underlying scheme $W^2_5(C) \simeq W^1_3(C)$ whose support consists of a single point. The birational map $C^{(3)} \dashrightarrow G^1_5(C)$ is a flip of threefold, and the embedding $\D(C^{(3)}) \hookrightarrow \D(G_5^1(C))$ is induced by the structure sheaf of the classical reducible scheme $C^{(3)}\times^{\rm cl}_{\Pic^3(C)} G_1^5(C)$.
\end{example}

\begin{example} 
\label{eg:curves:g=7}
If $C$ is a general trigonal curve of genus $7$, then $\dim W^1_6(C) = 3$, $\dim W^2_6(C) = 0$, and they are both nonempty (see \cite[Lemma 2.1]{Lar}).  In this case, $\bG^1_6(C) = G^1_6(C)$ is a classical, equidimensional scheme of dimension three. Hence Corollary \ref{cor:SOD:curves} implies a derived equivalence $\D(G^1_6(C)) \xrightarrow{\simeq}\D(G^1_6(C))$ for the threefold flop $G^1_6(C) \rightarrow W^1_6(C) \leftarrow G^1_6(C)$, where the second projection $G^1_6(C) \to W^1_6(C)\subseteq\Pic^6(C)$ is given by $(L,g^1_6) \mapsto L^\vee \otimes \omega_{C}$. Moreover, the derived equivalence is induced by a nonclassical incidence scheme whose underling scheme is the classical fiber products $G^1_6(C) \times_{\Pic^6(C)}^{\rm cl} G^1_6(C)$. \end{example}


The above examples highlight the importance of considering both the derived structures on $\bG_d^r$'s as well as derived structures on incidence correspondence schemes when studying semiorthogonal decompositions for varieties of linear series on special curves. 

\begin{remark}
\label{rmk:curves:Hilb}
The framework presented in this paper allows us to extend Corollary \ref{cor:SOD:curves} to families of singular curves $\sC/S$. For instance, consider a family $\sC/S$ of integral Gorenstein curves with arithmetic genus $g\geq 1$, and let $d\geq g-1$ be an integer. In the case of $r=0$, we obtain a semiorthogonal decomposition:
	$$\D(\Hilb^{d}_{\sC/S}) = \big \langle \D(\Hilb^{2g-2-d}_{\sC/S}) , ~ \D(\overline{{\rm Jac}}^d_{\sC/S})_{1}, \cdots,  \D(\overline{{\rm Jac}}^d_{\sC/S})_{1-g+d} \big \rangle,$$
where $\Hilb^{d}_{\sC/S}$ and $\Hilb^{2g-2-d}_{\sC/S}$ are derived Hilbert schemes of $d$ and $(2g-2-d)$ points, respecitvely, for the family $\sC/S$, and $\overline{{\rm Jac}}^d_{\sC/S}$ is the compactified Jacobian scheme parametrizing rank one torsion-free sheaves of degree $d$.
Similar generalizations of Corollary \ref{cor:SOD:curves} exist in the case of all $r$. The details will appear in a forthcoming paper.
\end{remark}

\bibliographystyle{alpha}
\bibliography{DQuot_refs}

\end{document}